\documentclass[11pt,a4paper]{amsart}
\usepackage[english]{babel}
\usepackage{amsmath}

\usepackage{amssymb}
\usepackage{graphicx}
\usepackage{epstopdf}

\usepackage{subfigure}
\usepackage{enumitem}

\usepackage[utf8]{inputenc} 
\usepackage[T1]{fontenc}
\usepackage{amsaddr}

\newcommand{\bpr}{\begin{trivlist} \item[]{\bf Proof. }}
\newcommand{\epr}{\hspace*{\fill} $\qed$\end{trivlist}}

\newcommand{\be}{\begin{eqnarray}}
\newcommand{\ee}{\end{eqnarray}}
\newcommand{\ba}{\begin{align}}
\newcommand{\ea}{\end{align}}
\newcommand{\bi}{\begin{itemize}}
\newcommand{\ei}{\end{itemize}}

\newcommand{\secref}[1]{Section~\ref{sec:#1}}

\newcommand{\seclab}[1]{\label{sec:#1}}
\newcommand{\eqlab}[1]{\label{eq:#1}}
\renewcommand{\eqref}[1]{(\ref{eq:#1})}

\newcommand{\figref}[1]{Fig.~\ref{fig:#1}}
\newcommand{\figlab}[1]{\label{fig:#1}}
\newcommand{\propref}[1]{Proposition~\ref{proposition:#1}}
\newcommand{\proplab}[1]{\label{proposition:#1}}
\newcommand{\lemmaref}[1]{Lemma~\ref{lemma:#1}}
\newcommand{\lemmalab}[1]{\label{lemma:#1}}
\newcommand{\corref}[1]{Corollary~\ref{cor:#1}}
\newcommand{\corlab}[1]{\label{cor:#1}}
\newcommand{\remref}[1]{Remark~\ref{remark:#1}}
\newcommand{\remlab}[1]{\label{remark:#1}}
\newcommand{\thmref}[1]{Theorem~\ref{theorem:#1}}
\newcommand{\thmlab}[1]{\label{theorem:#1}}

\newcommand{\defnlab}[1]{\label{defn:#1}}
\newcommand{\defnref}[1]{Definition~\ref{defn:#1}}

\newtheorem{theorem}{Theorem}[section]
\newtheorem{proposition}[theorem]{Proposition}
\newtheorem{definition}[theorem]{Definition}
\newtheorem{lemma}[theorem]{Lemma}
\newtheorem{cor}[theorem]{Corollary}

\newtheorem{remark}[theorem]{Remark}

\numberwithin{equation}{section}
\usepackage{fullpage}
\usepackage{color,comment,xcolor}
\definecolor{orange}{RGB}{255,127,0}

\begin{document}
\title{Circularization in the damped Kepler problem}
\author{K. U. Kristiansen and R. Ortega}



 \begin{abstract}
 {In this paper, we revisit the damped Kepler problem within a general family of nonlinear damping forces with magnitude $\delta \vert u\vert^{\beta}\vert \dot u\vert^{\alpha+1}$, depending on three parameters $\delta>0,\alpha\ge 0$ and $\beta\ge 0$, and address the general question of circularization whereby orbits tend to become more circular as they approach the sun. Our approach is based on dynamical systems theory, using blowup and desingularization as our main technical tools. We find that $\gamma=\alpha+2\beta-3$ is an important quantity, with the special case $\gamma=0$ separating circularization ($-3<\gamma<0$) where the eccentricity converges to zero, i.e. $e(t)\rightarrow 0$ as $u(t)\rightarrow 0$, from cases ($\gamma>0$) where $e(t)\rightarrow 1$ as $u(t)\rightarrow 0$, both on open sets of initial conditions. We find that circularization for $-3<\gamma<0$ occurs due to asymptotic stability of a zero-Hopf equilibrium point (i.e., the eigenvalues are $\pm i \omega,0$) of a three-dimensional reduced problem (which is analytic in the blowup coordinates). The attraction is therefore not hyperbolic and in particular not covered by standard dynamical systems theory. Instead we use recent results on normal forms of the zero-Hopf to locally bring the system into a form where the stability can be addressed directly. The case $\gamma=0$ relates to a certain scaling symmetry (that is also present in the {undamped} Kepler problem) and in this case the system can be reduced to a planar system. We find that the eccentricity limits to $4\delta^2\in (0,1)$ on an open subset of initial conditions for ($0<\delta<\frac12$ and $\gamma=0$). We also describe different  properties of the solutions, including finite time blowup and the limit of the eccentricity vector. Interestingly, we find that circularization $e(t)\rightarrow 0$ implies finite time blowup of solutions. We believe that our approach can be used to describe unbounded solutions. }
 %
 %

 \end{abstract}
 
  \maketitle

  \bigskip
\smallskip

\noindent \textbf{keywords.} {Blowup, invariant manifolds, damped Kepler problem, circularization, normal forms.}
\maketitle

\tableofcontents

\section{Introduction}
{All major planets in the Solar System {(Mercury being an exception)} have almost circular orbits. This fact is not predicted by the Keplerian theory and it has been attributed to the effect of some resistance force opposite to the velocity. In Cartesian coordinates, this effect can be described by the {damped} Kepler problem in the plane
\begin{align}\eqlab{keplerd}
\ddot{u}+\Delta(u,\dot{u})\frac{\dot{u}}{|\dot{u}|}=-\frac{u}{|u|^3},\; \; u\in \mathbb{R}^2 \setminus \{ 0\}.
\end{align}
{Here the function $\Delta$ is positive everywhere and models the type of friction}. We will focus on the following question: \textit{Under what conditions on $\Delta$ can we say that the orbits of \eqref{keplerd} have a tendency to become circular as they get close to the Sun?} This  {question} has a long tradition. A historical review, including the work of Tisserand and Encke, {is available} in the book
by See \cite{See}, published in 1910. Poincar\'e {also} considered this question in his Lessons on Cosmology (see \cite[chapter 6]{lecons} and \cite{poincare1912a} for an English translation). 
In all these classical works, it was assumed that the function $\Delta$ is of the type 
\begin{align}\eqlab{Dfunc}
\Delta(u,\dot u) = \delta \frac{\vert \dot u\vert^{\alpha+1}}{\vert u\vert^\beta},\quad \alpha,\beta\ge 0,\quad \delta>0;
\end{align}
  see \cite[Eq. (184)]{See} and \cite[Eq. (10), p. 122]{lecons}; when comparing with these references, please notice that we have shifted $\alpha$ by one unit. We will also accept this convention. {In \cite{lecons,See}, the equations \eqref{keplerd} were expressed in terms of the astronomical coordinates. Since the eccentricity $e$ is an astronomical coordinate, the notion of circularization was therefore naturally introduced as $e(t)\rightarrow 0$ for $t\rightarrow \omega$, where 
  \begin{align}\eqlab{collisiontime}
\mbox{$\omega$ is the time ($\omega\le \infty$)  of collision}.
\end{align}
In the present paper, we will define circularization completely  analogously, but $e(t)$ will not be a coordinate. Instead we express $e$ as a function of the phase space variables, $u$ and $\dot u$, for the equations in Cartesian coordinates, recall \eqref{keplerd}}. 
This approach has the advantage of extending to the complete phase space. See \secref{circularization} for more details.

{The problem of circularization was addressed by \cite{lecons,See} in a very direct way}. In particular, 
after expressing the eccentricity $e$ as a function of the true anomaly {$f$}, {i.e. $e=e(f)$}, the derivative $\frac{de}{df}$ was computed. In turns out that this derivative can change sign and {the goal of \cite{lecons,See} was to show that} the sequence of mean values over complete revolutions, $\frac{1}{2\pi}\int_{2\pi n}^{2\pi (n+1)} \frac{de}{df}{df}$, decreased to zero. The conclusions obtained by these methods are questionable because the formula for {$\frac{de}{df}$} is not based on the original differential equation, but on a modified equation obtained after averaging and truncation. {In the present paper, we will analyze the circularization problem rigorously using the modern machinery of dynamical systems.}

Our main result is concerned with all parameters $\alpha ,\beta \geq 0$, $(\alpha ,\beta )\neq (0,0)$ and we will prove that {circularization only occurs for}
\begin{align}
  -3<\gamma:=\alpha +2\beta -3<0.\eqlab{optimal}
\end{align}
In fact{, we find that} there exists an open and non-empty set of initial conditions leading to solutions with the circularization property when \eqref{optimal} holds. On the contrary, no solution can circularize if $ \gamma \geq 0$. The case $(\alpha ,\beta )= (0,0)$ has been excluded because it has been treated in previous papers and it is somehow exceptional. Circularization does occur (see \cite{hamilton2008a,haraux2021a,margheri2017a,margheri2017ab}) but only on a measure zero set of initial conditions (see also \cite{kristiansen2023revisiting}). The system for $\alpha =0$ has been analyzed in \cite{margheri2020a} and some of our conclusions were already obtained in that paper.
In particular the case $(\alpha ,\beta )= (0,2)$ has been analyzed in several papers because the equations can be integrated explicitly. See the references in \cite{margheri2020a}.

{Recently in \cite{kristiansen2023revisiting}, the first author used dynamical systems theory, with blowup, desingularization and normal forms as the main technical tools, to study the linear case $(\alpha ,\beta )=(0,0)$. The linear case is further distinguished by the presence of a conserved quantity (due to the eccentricity vector having a limit, see \cite{margheri2017a}), and the combination of blowup and normal form theory allowed the first author to solve some remaining questions regarding the smoothness of this first integral. We will use a similar approach in the present paper for $(\alpha ,\beta )\neq (0,0)$.}

Firstly, due to the invariance under rotations, {we reduce the {damped} Kepler problem  to a three-dimensional, first order system in the variables $(r,p,l)$ with $r=|u|$, $p=\dot{r}$, $l=|u\wedge \dot{u}|$.} In these coordinates, there is a singularity at $r=0$, but upon using appropriate time reparatrizations, we obtain a {\textit{desingularized} (and more regular)} vector field in the space $(r,v,l)\in [0,\infty )\times \mathbb{R}\times [0,\infty )$ with $v=pl$. In this way, a collision manifold $\{ r=0\}$ has been attached.
At first sight, this trick does not seem very useful because the vector field vanishes {(and is  degenerate)} on the boundaries ($r=0$) and ($l=0$). To overcome this difficulty, we follow \cite{kristiansen2023revisiting} and perform cylindrical blowup transformations of the lines $r=l=0$ and $r=v=0$: {
\begin{align*}
(\rho,"v",e^{i\phi}) \mapsto \begin{cases}
r&= \rho^2 \cos \phi,\\
v &="v",\\
l &=\rho \sin \phi,\end{cases}
\quad \mbox{resp.}\quad 
    (\rho,\mu ,e^{i\phi}) \mapsto \begin{cases}
r&= \rho^2,\\
v &=\mu \cos \phi,\\
l &=\rho \mu \sin \phi.\end{cases}
\end{align*}
(The former corresponds to \eqref{buV} whereas the latter one corresponds to {a combination of \eqref{buV} and} \eqref{buP2} below.) 
 The exponents are adjusted so that in the new variables, and after a time rescaling {(corresponding to division of the right hand side by a positive quantity for $\rho>0,\mu>0$)}, the vector field extends continuously and nontrivially (i.e. the vector field is not identically zero) to $\rho=0,\mu=0$; it is this process of time rescaling that is known as desingularization.} {In general, the most useful situation of blowup} is when the desingularization leads to a {smooth system having hyperbolic equilibria} within the boundary of the phase space, so that the usual hyperbolic methods (linearization, stable-, unstable- and center manifolds, etc) of dynamical systems theory, see e.g. \cite{perko2001a}, can be applied; see also \cite[Chapter 3.3]{dumortier2006a} for general results on blowup (including the use of Newton polygons to determine the weights) for planar systems and \cite{jelbart2021a,uldall2021a,kristiansen2020a} for the use of blowup to gain smoothness.  {This will in some ($\alpha,\beta$-dependent) cases require additional/successive blowup transformations in the present case.} 
{However, for the region of parameters ($\gamma<0$) with circularization,  the blowup approach will not lead to hyperbolicity (but rather ellipticity, as in \cite{kristiansen2023revisiting}) and this situation will be more delicate. In particular, we find that zero eccentricity  corresponds (in a certain sense) to a zero-Hopf equilibrium point and we prove asymptotic stability of this point (so that $e(t)\rightarrow 0$ on an open set of initial conditions) using recent results on convergence of analytic normal forms for the zero-Hopf.
}

\subsection{Overview}
{The paper is organized as follows: In \secref{circularization}, we first introduce the most basic concepts (angular momentum, eccentricity, etc) and then define our notion of circularization, see \defnref{circ}. Subsequently, in \secref{mainresult}, we present our main results, see \thmref{main}. 
The details of the blowup transformations, used to prove \thmref{main}, are given in \secref{blowup}. They will depend upon $\alpha$ and $\beta$. In particular, we divide the analysis into three main cases: $\gamma:=\alpha+2\beta-3>0$ (see \secref{gammaPos}), $\gamma<0$ (see \secref{gammaNeg}) and $\gamma=0$ (see \secref{gammaEq0}). For the convenience of the reader, \secref{gammaPos}--\secref{completing} -- where the main results are proven (through the proof of a series of propositions) -- can be read independently of the technical details of \secref{blowup}. We conclude the paper in \secref{discussion} with a discussion section.}

%
\section{The notion of circularization}\seclab{circularization}
%
{In principle the equation \eqref{keplerd} should be considered in the three dimensional space, but a simple computation on the angular momentum 
\begin{align}\eqlab{Ldefn}
L(u,\dot{u})=u\wedge \dot{u}, 
\end{align}
shows that }
\begin{align}
 \frac{d}{dt} L = -\Delta(u,\dot u) L, \eqlab{Leqn}
\end{align}
and the direction of $L\in \mathbb R^3$ is therefore preserved. Consequently, as in the {undamped} Kepler problem, corresponding to $\delta=0$, the dynamics with $L\ne 0$ are confined to an orbital plane (with normal vector $L$) and for this reason, we will consider \eqref{keplerd} on the phase space
$${\mathcal P}:=(\mathbb{R}^2 \setminus \{ 0\} )\times \mathbb{R}^2 $$
{with generic points $(u,\dot u)$.}

\subsection{The {undamped} system}
For $\delta=0$:
\begin{align}\eqlab{kepler0}
\ddot{u} =-\frac{u}{|u |^3},
\end{align}
corresponding to the {undamped} Kepler problem, 
the angular momentum $L$ is a conserved quantity and for $l:=\vert L\vert \ne 0$, the set 
%
$$\mathcal{C} =\{ u(t):\; t\in \mathbb{R} \} ,$$
with $u(t)$ a solution of \eqref{kepler0}, 
is a conic section (ellipse, parabola or {a branch of hyperbola}). 
 The conic can be described in Cartesian coordinates by the equation
\begin{equation}\eqlab{lag}
|u |+\langle \mathcal{E} ,u \rangle =l^2,\; \; u \in \mathbb{R}^2.
\end{equation}
In particular, the eccentricity vector $\mathcal{E}$ is given by the formula
\begin{align}
\mathcal{E} (u,\dot u)=\dot u\wedge L(u,\dot u)-\frac{u}{|u|},\eqlab{ecc}
\end{align}
and it is also a conserved quantity for \eqref{kepler0} on $\mathcal P$. The norm of $\mathcal{E}$,
$$e=|\mathcal{E} |,$$
is the eccentricity.
The circular motions of Kepler problem correspond to zero eccentricity, $e=0$. 





\subsection{The {damped} system}
Assume now that $u(t)=u(t;u_0,\dot u_0)$ is a solution of the {damped} Kepler problem \eqref{keplerd} with initial conditions $u(0)=u_0,\,\dot u(0)=\dot u_0$, $(u_0,\dot u_0)\in \mathcal P$, defined on a forward maximal interval $[0,\omega)$ where $\omega\in (0,\infty]$. The \textit {osculating conic} $\mathcal C=\mathcal C(t)$ {is defined as the conic that would correspond to the solution of the {undamped} problem passing through the point $(u,\dot u)(t)\in \mathcal P$. This curve evolves with time according to the equation  
\begin{equation} 
|u (t)|+\langle \mathcal{E}(t)  ,u(t) \rangle =l^2 (t),
\end{equation}
where $l(t)=\vert u(t)\wedge \dot u(t)\vert$ and $\mathcal{E}(t) =\mathcal{E}(u(t),\dot u(t))$, recall \eqref{Ldefn} and \eqref{ecc}. 
To justify this formula, it is useful to recall the following identity
$$
 |u|+\langle \mathcal{E} (u,\dot u), u\rangle =|L(u,\dot u)|^2 ,
$$ 
 valid for any point $(u,\dot u)$ in the phase space $\mathcal P$.
Here $\mathcal{E}$ and $L$ are interpreted as functions in the independent variables $x$ and $v$. The eccentricity of the osculating conic is
$$e(t)=| \mathcal{E} (u(t),\dot{u}(t))|.$$
}
\begin{definition}\defnlab{circ}
We say that there is circularization at $(u_0,\dot u_0)\in \mathcal P$ if
$$e(t)\to 0 \; \; \; {\rm as}\; \; t \rightarrow \omega^- ,$$
recall \eqref{collisiontime}.
Moreover, we say that \eqref{keplerd} has the circularization property on an open set if there is circularization on open set of initial conditions in $\mathcal P$.
\end{definition}
{In essence, this is the traditional notion of circularization, although the formulation may look different at first.
In classical textbooks (see for instance \cite[Section 172]{moulton1970introduction}), the function $e(t)$ is defined in terms of the astronomical coordinates and so the phase space must be reduced to the region of negative energy $E<0$ where the orbits of the {undamped} problem are ellipses. More precisely, to the subset of $\mathcal P$ composed by the points $(u,\dot u)$ such that $L(u,\dot u)\neq 0$ and
\begin{align*}
 E(u,\dot u):= \frac{1}{2}\vert \dot u\vert^2-\frac{1}{\vert u\vert} <0.
\end{align*}
Our formulation is based on Cartesian coordinates and is valid on the whole phase space $\mathcal P$.
}





\section{Main result}\seclab{mainresult}

Our main result can be summarized as follows. 
\begin{theorem}\thmlab{main}
{Consider \eqref{keplerd} with $\Delta$ given by \eqref{Dfunc} for $\delta>0$, $\alpha\ge 0$, $\beta\ge 0$, $\alpha +\beta >0$ and define \begin{align}\eqlab{gamma}\gamma:=\alpha+2\beta -3.
\end{align} 
Then some solution of \eqref{keplerd} has the circularization property if and only if $$-3<\gamma<0.$$
Moreover, given $\alpha$, $\beta$ and $\delta$, then there exists a non-empty, open set $\mathcal{U} \subset \mathcal P$ such that for every solution $u(t)$ with initial condition $(u(0),\dot{u}(0))\in \mathcal{U}$ the following holds:
\begin{enumerate}
\item If $\gamma\in (-3,0)$ then $\omega <\infty$, $u(t)\to 0$ and $e(t)\to 0$ as $t\to \omega$. Finally, $u(t)$ rotates infinitely many times around the origin for $t\in [0,\omega)$; more precisely, the argument $\theta(t)$ of $u(t)$ diverges, i.e. $\lim_{t\rightarrow \omega} \theta(t) =+\infty$. 
\item If $\gamma\ge 0$ then ($\omega=\infty$ if and only if $\alpha-\beta+3<0$),
  $u(t)\rightarrow 0$ and $\mathcal{E}(t)\to \mathcal{E}_{\omega}$ as $t\to \omega$ for some vector $\mathcal{E}_{\omega}$ on $S^1$, i.e. $|\mathcal{E}_{\omega}|=1$. Here $\mathcal{E}(t)$ denotes the eccentricity vector corresponding to $(u(t),\dot u(t))\in \mathcal P$. In addition, $u(t)$ rotates finitely many times around the origin; in particular, the argument $\theta(t)$ of $u(t)$ converges, i.e. $\lim_{t\rightarrow \omega} \theta(t)\in \mathbb R$ exists. 
   
  \item If $\gamma =0$ then, as $t\to \omega$, $u(t)\to 0$, $e(t)\to 4\delta^2$ if $\delta <\frac{1}{2}$ and $e(t)\to 1$ if $\delta \geq \frac{1}{2}$. 
  \end{enumerate}
  
}

\end{theorem}
{
\begin{remark}
    It is also possible to describe the behavior of the radial velocity for solutions with initial conditions in $\mathcal{U}$. Let $r(t):=\vert u(t)\vert$. Then in case (1),
     \begin{align*}
 \limsup_{t \rightarrow \omega} \left| \frac{d}{dt} r(t) \right| \begin{cases}
                     =\infty & \mbox{for $\alpha+\beta-1>0$},\\
                     \in (0,\infty) & \mbox{for $\alpha+\beta-1=0$},\\
                     =0 & \mbox{for $\alpha+\beta-1<0$}.
                    \end{cases}
\end{align*}
In case (2),
\begin{align*}
   \lim_{t\rightarrow \omega} \frac{d}{dt}r(t)= \begin{cases} 
                    \infty & \mbox{for $\beta<2$},\\
                   -\delta^{-\frac{1}{\alpha+1}} & \mbox{for $\beta=2$}, \\
                0 & \mbox{for $\beta>2$}.
                   \end{cases}
  \end{align*}
\end{remark}

We will refer to $\gamma=0$ as the critical case. It is distinguished by the following property:
\begin{lemma}\lemmalab{critical}
 Suppose that either $\gamma=0$ or $\delta=0$. Then 
for every $\mu,\lambda\in \mathbb R$ satisfying $\mu^2\lambda^3=1$, the following holds: If $t\mapsto u(t)$ is a solution of \eqref{keplerd} then so is $t \mapsto \lambda u(\mu t)$.
 \end{lemma}
\begin{proof}
 Simple calculation. 
\end{proof}

We summarize the results of \thmref{main} and the above remark in \figref{diagram}. Notice that these results are only valid on an open set of initial conditions, i.e. not on the whole phase space. In fact, we know that there exists unbounded solutions $u(t)\rightarrow \infty$ of \eqref{keplerd} in some cases. However, we believe that we can show that the open set in \thmref{main} coincides with the collision set where $u(t)\rightarrow 0$. But 
this requires more effort and we therefore leave this to future work. See \remref{r1zerocrit} for a discussion of this aspect in the context of the critical case $\gamma=0$.

\begin{figure}[h!]
\begin{center}
\includegraphics[width=.65\textwidth]{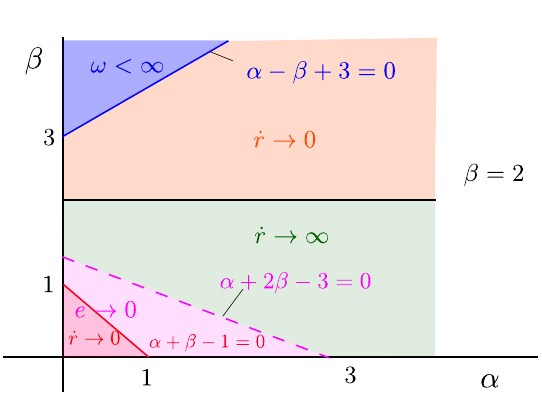}
\end{center}
\caption{Illustration of the results of \thmref{main}. Most importantly, \eqref{keplerd} only has the circularization property in the magenta region, strictly below the magenta and dashed line $\gamma=\alpha+2\beta-3<0$. Moreover, it is interesting to note that for $\gamma\ne 0$ global existence (i.e. $\omega=\infty$) 
{only occurs} within the green and orange region; {in the blue and magenta regions 
we have $\omega<\infty$.}}
\figlab{diagram}
\end{figure}

In order to prove \thmref{main}, we first identify $\mathbb R^2$ with $\mathbb C$ in the usual way and put 
\begin{align}
u=re^{i\theta},\eqlab{ueqn}
\end{align}
so that $r=\vert u\vert$. 
Then {it is not restrictive to assume}
\begin{align}\eqlab{leqn}
l:=\vert L\vert =r^2\dot \theta\ge 0.
\end{align}
 Moreover, \eqref{keplerd} becomes the first order system
\begin{equation}\eqlab{rttlt}
\begin{aligned}
\dot r &= p,\\
 \dot p &=-\frac{1}{r^2}+\frac{l^2}{r^3}-\delta \frac{\left(\frac{l^2}{r^2}+p^2\right)^{\frac{\alpha}{2}}}{r^\beta}p,\\
 \dot l&=-\delta \frac{\left(\frac{l^2}{r^2}+p^2\right)^{\frac{\alpha}{2}}}{r^\beta} l,
\end{aligned}
\end{equation}
for $r>0$. In these coordinates, we also have that 
\begin{align}
 e = l^2 p^2 +\frac{(l^2-r)^2}{r^2},\eqlab{ehere2}
\end{align}
using \eqref{ecc}. 
{In many cases,} $p$ is unbounded on trajectories. For this reason, it is useful to introduce 
\begin{align}\eqlab{veqn}
v = pl,
\end{align} such that 
\begin{equation}\eqlab{rvl0}
\begin{aligned}
 \dot r &=l^{-1} v,\\
 \dot v&= r^{-\alpha-\beta-3}l^{-\alpha-1} \left(- r^{\alpha+\beta+1}l^{\alpha+2}+r^{\alpha+\beta}l^{\alpha+4}  - 2\delta r^3 \left(l^4 +r^2 v^2\right)^{\frac{\alpha}{2}}lv \right),\\
 \dot l &=-\delta  r^{-\alpha-\beta} \left(l^4 +r^2 v^2\right)^{\frac{\alpha}{2}} l^{1-\alpha}.
\end{aligned}
\end{equation}
As in \cite{kristiansen2023revisiting},  this leads to a desirable compactification, insofar that $v$ {will} remain bounded (at least if $r\rightarrow 0$). {In these coordinates, we have} the following simple expression for the eccentricity
\begin{align}
 e = v^2 +\frac{(l^2-r)^2}{r^2}.\eqlab{Enorm}
\end{align}
This directly leads to the following characterization of circularization:
\begin{lemma} \lemmalab{circ} The initial condition $(u_0 ,\dot{u}_0 )\in \mathcal P$ with $L(u_0 ,\dot{u}_0 )\neq 0$ has the circularization property for \eqref{keplerd} if and only if the solution $(r,v,l)(t)$ of \eqref{rvl0}, with initial conditions $(r,v,l)(0)=(r_0,v_0,l_0)$, $r_0 =|u_0 |$, $l_0 =|L(u_0 ,\dot{u}_0 )|$, $v_0 =\frac{l_0}{r_0}\langle u_0 ,\dot{u}_0 \rangle$ satisfies $v(t)\rightarrow 0$ and $\frac{r(t)}{l(t)^2}\rightarrow 1$ for $t\rightarrow \omega$.
\end{lemma}
{An interesting consequence of this result is that circularization can only occur at collisions.}
{\begin{cor}\corlab{circ1}
    Assume that there is circularization for the solution $u(t)$. Then $r(t)\rightarrow 0$ as $t\rightarrow \omega$.
\end{cor}}
{\begin{proof}
    From the third equation in \eqref{rvl0} we see that $l(t)$ is decreasing and so the limit $l_{\omega}:=\lim_{t\rightarrow \omega} l(t) \geq 0$ exists. In view of the asymptotic equivalence of $r(t)$ and $l(t)^2$ it is enough to prove that $l_{\omega}=0$. Otherwise, as $t\rightarrow \omega$,  $$l(t)\to l_{\omega} >0,\; \; v(t)\to 0,\; \; r(t)\to l_{\omega}^2 .$$
    This implies $\omega =+\infty$ and therefore there is a sequence $t_n \to +\infty$ with $\dot{l}(t_n )\to 0$. Again we use the third equation in \eqref{rvl0} to pass to the limit and obtain a contradiction.
\end{proof}}

It therefore follows, that in order to analyze whether circularization occurs, we can study {solutions} of \eqref{rvl0} with $r(t),l(t)\rightarrow 0$. As advertised in the introduction, in order do so using dynamical systems theory, we first use desingularization to ensure that the right hand side of \eqref{rvl0} is well-defined for $r=0$. This frequently leads to degenerate equilibria points (where the linearization only has zero eigenvalues) and in the present case we will overcome this (and gain hyperbolicity/ellipticity) by using blowup, see \cite{dumortier_1996,dumortier2006a, krupa_extending_2001}. Moreover, recently in e.g. \cite{jelbart2021a,uldall2021a,kristiansen2020a} blowup has also been used to gain smoothness and this will also be important to us; \eqref{rvl0} clearly involves roots for general $\alpha,\beta\in \mathbb R$ and the blowup approach will allow us to select weights where smoothness can be gained. The blowup transformations that we use will be case dependent (in particular, they depend upon $\gamma$) and we will describe and motive these in the following section.   

We remind the reader that  \secref{gammaPos}--\secref{completing}, where the main results are proven, can be read independently of the technical details of \secref{blowup}. In particular, in each of these sections we summarize up the relevant transformations of \secref{blowup}.


\section{Blowup transformations}\seclab{blowup}
In order to study solutions of 
\eqref{rvl0} with $r(t),l(t)\rightarrow 0$  using dynamical systems--based techniques, we first multiply the right hand side of \eqref{rvl0} by $r^{\alpha+\beta+3}l^{\alpha+1} $. This leads to the system
\begin{equation}\eqlab{rvl}
\begin{aligned}
 r' &= r^{\alpha+\beta+3}l^{\alpha} v,\\
 v'&=- r^{\alpha+\beta+1}l^{\alpha+2}+ r^{\alpha+\beta} l^{\alpha+4}- 2\delta r^3 \left(l^4 +r^2 v^2\right)^{\frac{\alpha}{2}} lv,\\
 l' &=-\delta r^3 \left(l^4 +r^2 v^2\right)^{\frac{\alpha}{2}} l^2,
\end{aligned}
\end{equation}
where $r=l=0$ is now well-defined. 
Notice that the multiplication of the right hand side by $ r^{\alpha+\beta+3}l^{\alpha+1}$ corresponds to a transformation of time defined by 
\begin{align*}
 \frac{d\tau}{dt} = \frac{1}{r^{\alpha+\beta+3}l^{\alpha+1}} >0,
\end{align*}
with $\tau$ denoting the time used in \eqref{rvl}, i.e. $()'=\frac{d}{d\tau}$. We will introduce different reparametrization of times in the following. We will refer to each of these new times as $\tau$; while it is often not crucial for our purposes, the context should provide adequate information about how each of these times are related (implictly) to the original one $t$ used in \eqref{keplerd}. 

%
%

The advantage of working with \eqref{rvl} is that it  is well-defined on $r=0$ and $l=0$, which enables the use of local methods from dynamical systems theory to extract information about $r>0$ and $l>0$, where \eqref{rttlt}, \eqref{rvl0}, and  \eqref{rvl} are all equivalent.
 However, if $\alpha>0$, then $l=0$ defines a set of equilibria for \eqref{rvl}. By dividing the right hand side of \eqref{rvl} by $l^{\min(\alpha,1)}$, \eqref{rvl} is desingularized within $l=0$:
 \begin{align*}
     &\begin{cases} r' &=r^{\alpha+\beta+3}v,\\ v'&=0,
     \end{cases} \quad \mbox{for $\alpha\in [0,1]$}, 
     \end{align*}
     and \begin{align*}\begin{cases} r' &=0,\\v'&=-2\delta r^{3+\alpha} v^{\alpha+1},\end{cases}\quad \mbox{for $\alpha>1$}.
      \end{align*}
  Here we see that the sets $V$ and $R$ defined by $(0,v,0)$ and $(r,0,0)$ are degenerate sets of equilibria, in the sense that the linearization of \eqref{rvl} around any point in this set only has zero eigenvalues. ({Note that strictly speaking the linearization of \eqref{rvl} around points in $V$ only exists if $\alpha $ and $\alpha +\beta$ are not too small}; we will deal with this later). 
  To analyze \eqref{rvl} near the intersection of these sets, we will therefore follow the blowup procedure in \cite{kristiansen2023revisiting}. 



In further details, let 
\begin{align*}
 S^1 =\{(\bar r,\bar l)\in \mathbb R^2\,:\,\bar r^2+\bar l^2=1\},
\end{align*}
denote the unit circle. We then first blowup $V$ through a cylindrical blowup  transformation 
{\begin{align*}
\Phi^V\,: \,[0,\infty)\times\mathbb R\times  (S^1\cap \{\bar r\ge 0,\bar l\ge 0\}) &\rightarrow [0,\infty)\times \mathbb R \times [0,\infty),\\
(\rho,v,(\bar r,\bar l))&\mapsto (r,v,l),\end{align*}} defined by 
\begin{align}\eqlab{buV}
\Phi^V(\rho,v,(\bar r,\bar l))=\begin{pmatrix}
  \rho^2 \bar r\\
  v\\
  \rho \bar l                                                                                                                                                               \end{pmatrix}. 
\end{align}
Let $W$ denote the vector-field defined by \eqref{rvl} and recall that 
\begin{align*}
\gamma = \alpha+2\beta-3.
\end{align*}

\subsection{The case $\gamma>0$}\seclab{blowupPos}
In this case, we have the following.
\begin{lemma}
Let $\overline W=(\Phi^V)^*(W)$ denote the pull-back vector-field. Then the desingularized vector-field
$$\widetilde W:=\rho^{-2\alpha-7}\overline W,$$
extends continuously and nontrivially (i.e. $\widetilde W$ does not vanish identically along $\rho=0$) to $(\rho,v,(\bar r,\bar l))\in [0,\infty)\times\mathbb R\times  (S^1\cap \{\bar r\ge 0,\bar l\ge 0\})$.
\end{lemma}
\begin{proof}
We consider the $\bar r=1$-chart with chart-specific coordinates $(\rho_2,v,l_2)$, defined by
\begin{align*}
\Xi_2^V:\quad (\rho,v,(\bar r,\bar l)) \mapsto \begin{cases}
    \rho_2 &=\rho \sqrt{\bar r},\\
    v &= v,\\
    l_2 &= \frac{\bar l}{\sqrt{\bar r}},
 \end{cases}
\end{align*}
and obtain a local form of $\Phi^V$:
\begin{equation}\eqlab{Phi2V}
\begin{aligned}
\Phi_2^V:\quad (\rho_2,v,l_2)\mapsto \begin{cases}
 r &=\rho_2^2,\\
 v &= v,\\
 l&=\rho_2l_2,
 \end{cases}
\end{aligned}
\end{equation}
i.e. $\Phi^V = \Phi_2^V\circ \Xi_2^V$. This leads to the following local version $(\Phi_2^V)^* (W)$ of $\overline W$:
\begin{equation}\eqlab{V2eqns}
\begin{aligned}
 \rho_2' &=\rho_2^{2\alpha+7}  \left(\frac12 \rho_2^{\gamma+1} l_2^{\alpha} v\right),\\
 v'&=\rho_2^{2\alpha+7}  \left(-\rho_2^{\gamma}l_2^{\alpha+2}+ \rho_2^{\gamma} l_2^{\alpha+4}- 2\delta \left(l_2^4+v^2\right)^{\frac{\alpha}{2}}l_2 v\right),\\
 l_2'&=-\rho_2^{2\alpha+7} \left(\frac12 \rho_2^{\gamma} l_2^{\alpha-1} v+\delta \left(l_2^4+v^2\right)^{\frac{\alpha}{2}}\right)l_2^2.
\end{aligned}
\end{equation}
We see that $\rho_2^{2\alpha+7}$ is a common factor. By dividing the right hand side by this quantity, we obtain  -- using $\rho_2=\rho \sqrt{\bar r}$ -- a vector-field which is equivalent to the local version $\widetilde W_2=\rho^{-2\alpha-7}(\Phi_2^V)^* (W)$ of $\widetilde W$, as desired.  
Since $\gamma$ is assumed positive, $\rho_2=0$ defines an invariant set for \eqref{V2eqns} upon which we have
\begin{equation}\eqlab{this}
\begin{aligned}
  v'&=- 2\delta \left(l_2^4+v^2\right)^{\frac{\alpha}{2}}l_2 v,\\
 l_2'&=-\delta \left(l_2^4+v^2\right)^{\frac{\alpha}{2}}l_2^{2}.
\end{aligned}
\end{equation}
To complete the proof, we perform a similar calculation in the $\bar l=1$-chart (see \eqref{Phi1V} below). Notice that the $\bar r=1$- and $\bar l=1$- charts provide an atlas. We leave out further details.
\end{proof}
Notice that
$(v,l_2)=(0,0)$ is attracting for \eqref{this} on $l_2>0$; in fact the vector-field is equivalent to $v'=-2v$, $l_2'=-l_2$. However, the equilibrium is degenerate for \eqref{this} (all eigenvalues of the linearization are zero) and therefore we cannot apply local dynamical systems technique to describe the dynamics of
\eqref{V2eqns} in a neighborhood of the origin. In fact, the set $R_2$ defined by $(\rho_2,0,0)$ is a set of degenerate equilibria for {the rescaled version of} \eqref{V2eqns}, with the linearization around any point in $R_2$ having only zero eigenvalues.  This degeneracy stems from the degeneracy of the set $R:\,(r,0,0)$ of \eqref{rvl}. We therefore proceed to blow up $R_2$.

\subsubsection{Blowing up $R_2$}
To blow up $R_2$, we use a cylindrical blowup transformation:{\begin{align*}
\Phi^{R_2}\,:\,[0,\infty)\times [0,\infty)\times  (S^1\cap \{\bar l_2\ge 0\})&\rightarrow [0,\infty)\times \mathbb R\times [0,\infty),\\
(y,\mu,(\bar v,\bar l_2))&\mapsto (\rho_2,v,l_2),
\end{align*}
} defined by 
\begin{align}\eqlab{buP2}
\Phi^{R_2}(y,\mu,(\bar v,\bar l_2))=
\begin{pmatrix} 
 y^{\frac{1}{\gamma}}\\                                                                       
                       \mu \bar v\\
                       \mu \bar l_2\end{pmatrix}
.\end{align}
%
The points in $S^1$ are now of the type $(\bar v,\bar{l}_2 )$.
The transformation of the $\rho_2$ component is introduced to gain smoothness of $\rho_2=0$, see \eqref{yv1mu1}.
\begin{lemma}
 Let $\overline{\widetilde W}_2=(\Phi^{R_2})^*(\widetilde W_2)$ with $\widetilde W_2$ given by \eqref{V2eqns}. Then the desingularized vector-field
$$\widetilde{\widetilde W}_2:=\mu^{-\alpha-1}\overline{\widetilde W}_2,$$
extends continuously and nontrivially (i.e. $\widetilde{\widetilde W}_2$ does not vanish identically along $\mu=0$) to $(y,\mu,(\bar v,\bar l_2))\in [0,\infty)\times[0,\infty)\times  (S^1\cap \{\bar l_2\ge 0\})$. 
\end{lemma}
   
\begin{proof}
We will only consider the $\bar l_2=1$-chart with chart-specific coordinates $(y,v_1,\mu_1)$ so that $\Phi^{R_2}$ takes the local form
\begin{equation}\eqlab{mu2l22}
\Phi_1^{R_2}:\,(y,v_1,\mu_1)\mapsto \begin{cases}
\rho_2 &=y^{\frac{1}{\gamma}},\\
 v &= \mu_1v_1,\\
 l_2 &=\mu_1.
\end{cases}
\end{equation}
This leads to the vector-field $\mu_1^{-\alpha-1} (\Phi_1^{R_2})^*(\widetilde W_2)$:
\begin{equation}\eqlab{yv1mu1}
\begin{aligned}
 y' &=\frac{\gamma}{2}y^2 v_1,\\
 v_1'&=-y+\mu_1^2 y+\frac12 v_1^2 y -\delta\left(\mu_1^2+v_1^2\right)^{\frac{\alpha}{2}} v_1,\\
 \mu_1'&=-\left(\frac12 y v_1+\delta\left(\mu_1^2+v_1^2\right)^{\frac{\alpha}{2}} \right)\mu_1,
 \end{aligned}
 \end{equation}
  which is equivalent to the local version $\widetilde{\widetilde W}_{21}=\mu^{-\alpha-1}(\Phi_1^{R_2})^*(\widetilde W_2)$. 
\end{proof}
$y=0$ defines an invariant set of \eqref{yv1mu1} and within this set, we have 
\begin{align*}
 v_1' &=-\delta\left(\mu_1^2+v_1^2\right)^{\frac{\alpha}{2}} v_1,\\
 \mu_1'&=-\delta\left(\mu_1^2+v_1^2\right)^{\frac{\alpha}{2}}\mu_1.
\end{align*}
Here $(v_1,\mu_1)=(0,0)$ is attracting, but still degenerate (and {possibly not very smooth} at $\mu_1=v=0$) for $\alpha>0$. 

For $\alpha>0$, we therefore apply the following spherical blowup transformation of $(y,v_1,\mu_1)=(0,0,0)$:
\begin{align}\eqlab{this2}
 q\ge 0,(\bar y,\bar v_1,\bar \mu_1)\in S^2\mapsto \begin{cases}
                                                         y &=q^{\alpha+1}\bar y,\\
                                                         v_1 &=q \bar v_1,\\
                                                        \mu_1 &=q\bar \mu_1,
                                                       \end{cases}
\end{align}
where $S^2=\{(\bar y,\bar v_1,\bar \mu_1)\in \mathbb R^3\,:\,\bar y^2+\bar v_1^2+\bar \mu_1^2=1\}$. 

We now proceed in a more ad-hoc manner. 
Notice that the weights $(\alpha+1,1,1)$ are so that $\mu_1$ and $v_1$, as well as 
\begin{align*}
 -y\quad \text{and}\quad -{\delta} \left(\mu_1^2+v_1^2\right)^{\frac{\alpha}{2}} v_1,
\end{align*}
in the equation for $v_1$,
have the same order with respect to $q$. 

It is enough for our purposes to focus on the $\bar y=1$-chart with chart-specific coordinates $(q_1,v_{11},\mu_{11})$ so that \eqref{this2} takes the local form
\begin{equation}\eqlab{final}
\begin{aligned}
 y &= q_1^{\alpha+1},\\
 v_1&=q_1 v_{11},\\
 \mu_1 &=q_1\mu_{11}.
\end{aligned}
\end{equation}
Inserting this into \eqref{yv1mu1}, we obtain
\begin{equation}\eqlab{beta1V2P1eqnsbu}
\begin{aligned}
 q_1' &=\frac{\gamma}{2(\alpha+1)}q_1^{3} v_{11},\\
 v_{11}'&=-1+q_1^{2} \left(\mu_{11}^2+\frac{2-\beta}{\alpha+1}v_{11}^2\right)-\delta \left(\mu_{11}^2+v_{11}^2\right)^{\frac{\alpha}{2}} v_{11},\\
 \mu_{11}' &=-\left(\frac{\alpha+\beta-1}{\alpha+1} q_1^{2} v_{11}+\delta \left(\mu_{11}^2+v_{11}^2\right)^{\frac{\alpha}{2}}\right)\mu_{11},
\end{aligned}
\end{equation}
after division of the right hand side by $q_1^{\alpha}$. 

Now, notice that $q_1=0$ defines an invariant set for \eqref{beta1V2P1eqnsbu} and along this set we have
\begin{equation}\eqlab{beta1V2P1eqnsbumu10}
\begin{aligned}
 v_{11}' &=-1-\delta \left(\mu_{11}^2+v_{11}^2\right)^{\frac{\alpha}{2}}v_{11},\\
 \mu_{11}'&=-\delta \left(\mu_{11}^2+v_{11}^2\right)^{\frac{\alpha}{2}}\mu_{11}.
\end{aligned}
\end{equation}
Define
\begin{align}
v_{11,*}:=-{\delta^{-\frac{1}{\alpha+1}}}. \eqlab{pstar}
\end{align}
Then a simple calculation shows that $(v_{11,*},0)$ is a hyperbolic equilibrium for \eqref{beta1V2P1eqnsbumu10}; in particular, \eqref{beta1V2P1eqnsbumu10} is smooth in a neighborhood of this point. We have therefore gained hyperbolicity through blowup. We study the dynamics in details in \secref{gammaPos}.

\subsection{The case $\gamma\le 0$}\seclab{blowupNeg}
We now turn to $\gamma\le 0$ and the use of the blowup $\Phi^V$ in this case. It is straightforward to show that $\rho^{\gamma+2\alpha+7}=\rho^{3\alpha+2\beta+4}$ is a common factor of $\overline W=(\Phi^V)^*(W)$. We therefore study $\widetilde W:=\rho^{-3\alpha-2\beta-4} \overline W$ for $\gamma\le 0$.

However, in the case $\gamma\le 0$ it will be more useful to consider a separate chart $\bar l=1$, instead of $\bar r=1$, see \eqref{Phi2V}, with chart-specific coordinates $(\rho_1,r_1)$ defined by
\begin{align*}
\Xi_1^V:\quad (\rho,(\bar r,\bar l)) \mapsto \begin{cases}
  \rho_1 &=\rho \bar l,\\
  r_1 &= \bar l^{-2}{\bar r}.
 \end{cases}
\end{align*}
These gives the local form of $\Phi^V$:
\begin{equation}\eqlab{Phi1V}
\begin{aligned}
\Phi_1^V:\quad (\rho_1,r_1)\mapsto \begin{cases}
 r &=\rho_1^2 r_1,\\
 l&=\rho_1,
 \end{cases}
\end{aligned}
\end{equation}
i.e. $\Phi^V = \Phi_1^V\circ \Xi_1^V$,
and the following local system 
\begin{equation}\eqlab{V1eqnsgammaNegative0}
\begin{aligned}
 r_1 '& = v + 2\delta \rho_1^{-\gamma} r_1^{-\alpha-\beta} \left(1+r_1^2v^2\right)^{\frac{\alpha}{2}}r_1,\\
 v'&=-\frac{r_1-1}{r_1^3}-2\delta \rho_1^{- \gamma} r_1^{-\alpha-\beta} \left(1+r_1^2v^2\right)^{\frac{\alpha}{2}}v,\\
 \rho_1'&=- \delta r_1^{-\alpha-\beta} \left(1+r_1^2v^2\right)^{\frac{\alpha}{2}}\rho_1^{- \gamma+1}.
\end{aligned}
\end{equation}
Here we have divided the right hand side of $(\Phi_1^V)^* (W)$ by  $\rho_1^{3\alpha+2\beta+4} r_1^{{\alpha +\beta +3}}$ to simplify the equations further.

The following equations 
\begin{align}
 r_1 =\frac{1}{l_2^2},\quad \rho_1 = \rho_2 l_2,\eqlab{cc12}
\end{align}
define a smooth change coordinates between the two charts $\bar l=1$ and $\bar r=1$ for $l_2>0$.

For $\gamma<0$, in order to gain smoothness at $\rho_1=0$, we also transform $\rho_1$ in the following way:
\begin{align}
 x = \rho_1^{-\gamma}.
\end{align}
Then we finally obtain
\begin{equation}\eqlab{V1eqnsgammaNegative}
\begin{aligned}
 r_1 '& = v + 2\delta x r_1^{-\alpha-\beta} \left(1+r_1^2v^2\right)^{\frac{\alpha}{2}}r_1,\\
 v'&=-\frac{r_1-1}{r_1^3}-2\delta x r_1^{-\alpha-\beta} \left(1+r_1^2v^2\right)^{\frac{\alpha}{2}}v,\\
 x'&= \gamma \delta  r_1^{-\alpha-\beta} \left(1+r_1^2v^2\right)^{\frac{\alpha}{2}}x^2.
\end{aligned}
\end{equation}
An astonishing characteristic of this system is that \textit{it is analytic}, even along $x=0$ (corresponding to $\rho_1=l=0$) for all $\gamma<0$. 
%



%



In the following sections, we use the coordinates obtained through the blowup process to study $\gamma>0$, $\gamma<0$ and finally $\gamma=0$.
\section{{Absence} of circularization for $\gamma>0$}\seclab{gammaPos}
In this section, we will prove the following
\begin{proposition}\proplab{gammaPos}
 Suppose that  $\gamma=\alpha+2\beta-3> 0$. Then the following claims hold true:
 {\begin{enumerate}
     \item There is no circularization for any initial condition $(u_0 ,\dot{u}_0 )\in \mathcal P$
     \item There exists an open and non-empty set $\mathcal{U}\subset \mathcal P$ such that for each $(u_0 ,\dot{u}_0 )\in \mathcal{U}$ and $t\to \omega$,
     $$\mathcal{E}(t)\to \mathcal{E}_{\omega}, \; \; |\mathcal{E}_{\omega} |=1.$$
     Here $\mathcal{E}(t)$ is the eccentricity vector corresponding to $u(t;u_0 ,\dot{u}_0 )$ and $\mathcal{E}_{\omega}$ is some vector on the unit circle.
 \end{enumerate}}
\end{proposition}

Following \secref{blowup}, specifically \secref{blowupPos},  we divide the analysis into two parts: $\alpha=0$ and $\alpha>0$. 
\subsection{The case $\alpha=0$}\seclab{gammaPosAlpha0}
We consider \eqref{rttlt} and use a chain of changes of variables that can be summed up by the formulas:
\begin{align}\eqlab{trans}
 r= y^{\frac{2}{\gamma}},\quad p = y^{-\frac{1}{\gamma}} v_1,\quad l = y^{\frac{1}{\gamma}}\mu_1,\quad \frac{d\tau}{dt}=y^{-\frac{2\beta}{\gamma}},
\end{align}
see \eqref{veqn}, \eqref{Phi2V} and \eqref{mu2l22}.
This gives the equations \eqref{yv1mu1}$_{\alpha=0}$, repeated here for convenience
\begin{equation}\eqlab{yv1mu10}
\begin{aligned}
 y' &=\frac{\gamma}{2}y^2 v_1,\\
 v_1'&=-y+\mu_1^2 y+\frac12 v_1^2 y -\delta v_1,\\
 \mu_1'&=-\left(\frac12 yv_1 +\delta\right)\mu_1,
 \end{aligned}
 \end{equation}
with $()'=\frac{d}{d\tau}$. In these coordinates, we have
\begin{align}
 e = 1-2\mu_1^2\left(1-\frac12 (\mu_1^2+v_1^2)\right).\eqlab{ehere}
 \end{align}
We are now ready to prove the first claim of \propref{gammaPos}  in the case $\alpha =0$. Assume by contradiction that there is circularization for some solution. From \lemmaref{circ} and \corref{circ1}, for $t \to \omega$,
$$r(t)\to 0,\; \; \; \frac{r(t)}{l(t)^2}\to 1,\; \; \; e(t)\to 0.$$ The corresponding solution of \eqref{yv1mu10} will be defined on some interval $[0,T)$ and we know from
\eqref{trans} and \eqref{ehere} that, for $\tau \to T$,
$$y(\tau )\to 0,\quad v_1 (\tau )\to 0,\quad  \mu_1 (\tau )\to 1.$$ This is impossible. The case $T=\infty$ can be discarded because $(0,0,1)$ is not an equilibrium of \eqref{yv1mu10}. If $T<\infty$ there is a contradiction with the invariance of $\{ y=0\}$ under the flow associated to \eqref{yv1mu10}.

The proof of the second claim will be more delicate and will depend on the analysis of the local dynamics around equilibria of \eqref{yv1mu10}.
Now, \eqref{yv1mu10} has a single equilibrium at $(y,v_1,\mu_1)=(0,0,0)$. The linearization has eigenvalues $0,-\delta,-\delta$. Consequently, by center manifold theory (see e.g. \cite{car1}) we obtain the following.
\begin{lemma}\lemmalab{111}
Fix any $k\in \mathbb N$. Then there exists a one-dimensional attracting center manifold $W^c$ of $(y,v_1,\mu_1)=(0,0,0)$ for \eqref{yv1mu10}.  In particular, locally $W^c$ takes the following $C^k$-smooth graph form
 \begin{align}
  v_1 = -\frac{1}{\delta} y\left(1+\mathcal O(y)\right),\quad \mu_1=G(y),\eqlab{WC}
 \end{align}
  over $y\in [0,y_0]$, $y_0>0$ small enough; here $G:[0,y_0]\rightarrow [0,\infty)$ is a $C^k$-flat function so that $G^{(n)}(0)=0$ for any $0\ne n\le k$.
\end{lemma}
\begin{proof}
 Simple calculation.
\end{proof}
On $W^c$, we find 
\begin{align}
 y' = -\frac{\gamma}{2\delta}y^3  \left(1+\mathcal O(y)\right).\eqlab{yEqnCM1}
\end{align}
Consequently, $(y,v_1,\mu_1)=(0,0,0)$ is an attractor for \eqref{yv1mu10} and the center manifolds are nonunique. The rest of the proof of \propref{gammaPos} will be presented below together with the case of positive $\alpha$.

\subsection{The case $\alpha>0$}
In this case, we use a chain of changes of variables that can be summed up by the formulas:
\begin{align}\eqlab{trans1}
 r= q_1^{\frac{2(1+\alpha)}{\gamma}},\quad p = q_1^{\frac{2\beta-4}{\gamma}}  v_{11},\quad l = q_1^{\frac{1+\alpha+\gamma}{\gamma}}\mu_{11},\quad \frac{d\tau}{dt}=q_1^{-\frac{2(2\alpha+\beta)}{\gamma}},
\end{align}
see \eqref{veqn}, \eqref{Phi2V}, \eqref{mu2l22} and \eqref{final}.
This gives the equations \eqref{beta1V2P1eqnsbu}, repeated here for convenience
\begin{equation}\eqlab{beta1V2P1eqnsbu1}
\begin{aligned}
 q_1' &=\frac{\gamma}{2(\alpha+1)}q_1^{3} v_{11},\\
 v_{11}'&=-1+q_1^{2} \left(\mu_{11}^2+\frac{2-\beta}{\alpha+1}v_{11}^2\right)-\delta \left(\mu_{11}^2+v_{11}^2\right)^{\frac{\alpha}{2}} v_{11},\\
 \mu_{11}' &=-\left(\frac{\alpha+\beta-1}{\alpha+1} q_1^{2} v_{11}+\delta \left(\mu_{11}^2+v_{11}^2\right)^{\frac{\alpha}{2}}\right)\mu_{11}.
\end{aligned}
\end{equation}
with $()'=\frac{d}{d\tau}$.  In these coordinates, we have
\begin{align}
 e = 1-2q_1^2\mu_{11}^2\left(1-\frac12q_1^2 (\mu_{11}^2+v_{11}^2)\right).\eqlab{ehere1}
\end{align}
We can now complete the proof of the first claim of \propref{gammaPos}. We proceed as before and assume by contradiction the existence of a circularizing solution. The corresponding solution of \eqref{beta1V2P1eqnsbu1} satisfies, for $\tau \to T$, $$q_1 (\tau )\to 0,\; \; \; q_1 (\tau )^2 \mu_{11} (\tau )^2 \to 1,\; \; \; q_1 (\tau )^2 v_{11} (\tau )^2 \to 0.$$ In consequence, 
$q_1 (\tau )^2 |v_{11} (\tau )|\leq q_1 (\tau )^2 (1+ v_{11} (\tau )^2)  \to 0$. From the third equation of \eqref{beta1V2P1eqnsbu1} it is possible to obtain a differential inequality of the type $$\mu_{11}'\leq -(\epsilon (t)+\delta \mu_{11}^{\alpha} )\mu_{11},\; \; \; \mu_{11} >0,$$
with $\lim_{\tau \to T} \epsilon (\tau )=0$. For $\tau \to T$ we know that $q_1 (\tau )^2 \mu_{11} (\tau )^2 \to 1$, then $\mu_{11}(\tau )\to \infty$. This is incompatible with the above differential inequality.  \par
As in the previous case we study the local behavior of the system near the equilibrium.
Now, $(q_1,v_{11},\mu_{11})=(0,v_{11,*},0)$, with $v_{11,*}$ given by \eqref{pstar} repeated here for convenience
\begin{align}
v_{11,*}:=-{\delta^{-\frac{1}{\alpha+1}}}, \nonumber
\end{align}
is the unique equilibrium for \eqref{beta1V2P1eqnsbu1}. A simple computation shows that the linearization has eigenvalues
\begin{align*}
 0,\,-\delta^{\frac{1}{1+\alpha}},\,-(1+\alpha)\delta^{\frac{1}{1+\alpha}}.
\end{align*}
\begin{lemma}\lemmalab{aaa}
Fix any $k\in \mathbb N$. Then there exists a one-dimensional center manifold $W^c$ of $(q_1,v_{11},\mu_{11})=(0,v_{11,*},0)$ for \eqref{beta1V2P1eqnsbu1}. In particular, locally $W^c$ takes the following $C^k$-smooth graph form
 \begin{align*}
  v_{11} = v_{11,*}(1+\mathcal O(q_1^2)),\quad \mu_{11}=G(q_1)
 \end{align*}
 over $q_1\in [0,q_{1,0}]$, $q_{1,0}>0$ small enough; here $G:[0,y_0]\rightarrow [0,\infty)$ is a $C^k$-flat function so that $G^{(n)}(0)=0$ for any $0\ne n\le k$.
\end{lemma}\lemmalab{222}
\begin{proof}
 Simple calculation. Note that the vector field in \eqref{beta1V2P1eqnsbu1} is analytic in a neighborhood of the equilibrium.
\end{proof}
On $W^c$, we have the following reduced problem
\begin{align}
 q_1' &=\frac{\gamma}{2(\alpha+1)}q_1^{3}v_{11,*}(1+\mathcal O(q_1^2)),\eqlab{q1eqn}
\end{align}
Since $v_{11,*}<0$, $(q_1,v_{11},\mu_{11})=(0,v_{11,*},0)$ is an attractor for \eqref{beta1V2P1eqnsbu1} and the center manifold is nonunique. 

{It is now possible to define the set $\mathcal{U}$ appearing in the second claim of \propref{gammaPos}. Let $\mathcal{V} \subset \mathbb{R}^3$ be the region of attraction of the equilibrium for \eqref{beta1V2P1eqnsbu1}. In the case $\alpha =0$ we should define $\mathcal{V}$ as the attracting region of the origin for \eqref{yv1mu10}. The asymptotic stability of the equilibrium implies that $\mathcal{V}$ is open. Consider $\mathcal{V}_+ =\mathcal{V} \cap \{ q_1 >0\}$. The change of variables defined by \eqref{trans1} will transform $\mathcal{V}_+$ in an open subset $\mathcal{W}$ of $\mathbb{R}^3 \cap \{ r>0\}$. Similar arguments allow us to define $\mathcal{W}$ in the case $\alpha =0$. Finally, the set $\mathcal{U}\subset \mathcal P$
is obtained from $\mathcal{W}$ via the action of the group of rotations in the phase space. Precisely,
$$\mathcal{U}=\{ (u, \dot{u})\in \mathcal P:\; u=re^{i\theta},\, \dot{u}=pe^{i\theta}+\frac{l}{r} ie^{i\theta} ,\; (r,p,l)\in \mathcal{W},\; \theta \in \mathbb{R} \}.$$
This set is open in $\mathbb{R}^4$.}



\begin{lemma}\lemmalab{timeblowup1}
Suppose that $\gamma>0$ and that $(r(0),p(0),l(0))\in \mathcal{W}$. Then $\omega=\infty$ if and only if $\alpha-\beta+3 \leq 0$.
\end{lemma}
\begin{proof}
We only focus on the $\alpha>0$ case; the case $\alpha=0$ can be studied in the same way using the results from \secref{gammaPosAlpha0}, see also \remref{rem:margheri}.
We know that the corresponding solution of \eqref{beta1V2P1eqnsbu1} is defined on $[0,\infty )$ and $(q_1,v_{11},\mu_{11})(\tau)\rightarrow (0,v_{11,*},0)$ for $\tau\rightarrow \infty$, and since the flow near $(q_1,v_{11},\mu_{11})=(0,v_{11,*},0)$ is described by the flow on a center manifold, we only need to consider \eqref{q1eqn}. This gives
  \begin{align}\eqlab{qineq}
  q_1(\tau) \sim \frac{1}{\sqrt{\tau}},
 \end{align}
 for all $\tau\ge 1$, say. 
Here and in the following, we use $\sim$ to indicate that there exists constants $0<c_1<c_2$ such that the left hand side of \eqref{qineq} is bounded from below (above) by $c_1$ ($c_2$, respectively) times the right hand side of \eqref{qineq}, i.e.
\begin{align*}
 \frac{c_1}{\sqrt{\tau}} \le q_1(\tau)\le \frac{c_2}{\sqrt{\tau}},
\end{align*}
for all $\tau\ge 1$.
Therefore by using \eqref{trans1}, we find that
 \begin{align}
  \frac{d\tau}{dt}\sim \tau^{\frac{2\alpha+\beta}{\gamma}},\eqlab{dtau_dt1}
 \end{align}
for all $\tau\ge 1$. Finally, from \eqref{dtau_dt1},
$$\omega =t_1 +\int_1^\infty \frac{dt}{d\tau}(\tau ) d\tau \leq c_3 \left(1+\int_1^{\infty} \frac{d\tau}{\tau^{\frac{2\alpha +\beta}{\gamma}}}\right)<\infty,$$
 if and only if
\begin{align*}
 \frac{2\alpha+\beta}{\alpha+2\beta-3}>1\Leftrightarrow \alpha-\beta+3>0,
\end{align*}
and the result therefore follows.
\end{proof}
\begin{lemma}\lemmalab{p1}
Consider the same assumptions as in \lemmaref{timeblowup1}. Then
\begin{align*}
 \lim_{t\rightarrow \omega} p(t) =\begin{cases}
                  \infty &\mbox{for $\beta<2$},\\
                  v_{11,*} & \mbox{for $\beta=2$},\\
                  0 & \mbox{for $0\le \beta<2$}.
                 \end{cases}
\end{align*}
\end{lemma}
\begin{proof}
Consider $\alpha>0$ (the case $\alpha=0$ can be studied in the same way using the results from \secref{gammaPosAlpha0}). {For $\beta =2$, $v_{11,*} =-\frac{1}{\delta}$}.
The result then follows from (a) $(q_1,v_{11},\mu_{11})(\tau)\rightarrow (0,v_{11,*},0)$ for $\tau\rightarrow \infty$ and (b) the fact that the exponent on $q_1$ in {the second identity of} \eqref{trans1} for $\gamma>0$ {satisfies}
\begin{align*}
 \frac{2\beta-4}{\gamma} \lessgtr 0 \Leftrightarrow \beta \lessgtr 2.
\end{align*}
\end{proof}
\begin{lemma}\lemmalab{theta1}
Consider the same assumptions as in \lemmaref{timeblowup1}. Then 
\begin{align*}
 \lim_{t\rightarrow \omega} \theta(t),
\end{align*}
exists.
\end{lemma}
\begin{proof}
We consider $\alpha>0$; again $\alpha=0$ is similar and therefore left out. 
  We have $l=r^2 \dot{\theta}$, $\dot\theta=\frac{d}{dt}\theta$. Then by \eqref{trans1}  it follows that
$$\frac{d\theta}{d\tau} =\frac{d\theta}{dt} \frac{dt}{d\tau}=\mu_{11} q_1^2.$$
Since $\mu_{11}(\tau)\rightarrow 0$ exponentially (due to the center manifold $W^c$, see also \eqref{beta1V2P1eqnsbu1}), the limit
$$\lim_{t\to  \omega^-} \theta (t)=\theta (0)+\int_0^\infty \frac{d\theta}{d\tau}(\tau )d\tau,$$ is finite.
\end{proof}
\begin{remark}\remlab{rem:margheri}
 {In the case $\alpha =0$, the properties stated in the previous lemmas were already found in \cite[Theorem 6]{margheri2020a}. The result in that paper was valid for all possible collision solutions. }
\end{remark}

\begin{lemma}\lemmalab{ecc1}
Consider the same assumptions as in \lemmaref{timeblowup1}. Then the eccentricity vector $\mathcal E(t)\in \mathbb C$ has a limit as $t\rightarrow \omega$:
\begin{align*}
    \lim_{t\rightarrow \omega} \mathcal E(t) = -e^{i\lim_{t\rightarrow \omega} \theta(t)}. 
\end{align*}
\end{lemma}
\begin{proof}
We have
\begin{align*}
\vert \dot u\wedge L(u,\dot u)\vert \le \vert p\vert l + \frac{l^2}{r} = \mu_1 |v_1 |+\mu_1^2\rightarrow 0,
\end{align*}
in the case $\alpha=0$ and 
\begin{align*}
   \vert \dot u\wedge L(u,\dot u)\vert \le q_1^2|v_{11}|\mu_{11} + q_{1}^2 \mu_{11}^2\rightarrow 0, 
\end{align*}
in the case $\alpha>0$ as $\tau\rightarrow \infty$. The result therefore follows from \lemmaref{theta1} and the formula for $\mathcal E$, see \eqref{ecc}.
\end{proof}
\section{Circularization for $-3<\gamma<0$}\seclab{gammaNeg}
In this section, we will prove the following.
\begin{proposition}\proplab{gammaNeg}
 Suppose that  $-3<\gamma< 0$. Then \eqref{keplerd} has the circularization property on an open and non-empty set of initial conditions. 
\end{proposition}

For simplicity, we put $$\tilde \gamma:=-\gamma,$$ henceforth. Then $0<\tilde \gamma<3$. 
Following \secref{blowup}, specifically \secref{blowupNeg}, we consider \eqref{rvl} and use a chain of changes of variables that can be summed up by the formulas:
\begin{align}\eqlab{trans2}
 r= x^{\frac{2}{\tilde \gamma}} r_1,\quad p=\frac{v}{x^{\frac{1}{\tilde \gamma}}},\quad l=x^{\frac{1}{\tilde\gamma}},\quad \frac{d\tau}{dt}=\frac{1}{x^{\frac{3}{\tilde \gamma}}}.
\end{align}
This gives \eqref{V1eqnsgammaNegative0}, repeated here for convenience
\begin{equation}\eqlab{V1eqnsgammaNegative1}
\begin{aligned}
 r_1 '& = v + 2\delta x r_1^{-\alpha-\beta} \left(1+r_1^2v^2\right)^{\frac{\alpha}{2}}r_1,\\
 v'&=-\frac{r_1-1}{r_1^3}-2\delta x r_1^{-\alpha-\beta} \left(1+r_1^2v^2\right)^{\frac{\alpha}{2}}v,\\
 x'&=-\tilde \gamma \delta  r_1^{-\alpha-\beta} \left(1+r_1^2v^2\right)^{\frac{\alpha}{2}}x^2,
\end{aligned}
\end{equation}
where $()'=\frac{d}{d\tau}$. 
%
%
%
Within $x=0$, we have 
\begin{equation}\eqlab{r1vHam}
\begin{aligned}
 r_1 ' &= v,\\
 v'&= -\frac{r_1-1}{r_1^3}.
\end{aligned}
\end{equation}
 This system is Hamiltonian with Hamiltonian function:
\begin{align*}
 H(r_1,v) = \frac12 v^2+\frac{(r_1-1)^2}{2r_1^2}.
\end{align*}
In particular, $(r_1,v)=(0,0)$ is a center (the linearization having eigenvalues $\pm i$), surrounded by periodic orbits $\Gamma_1\!(h)$, $h\in ]0,\frac12[$, given by the level curves $H(r_1,v)=h$. The orbit $\Gamma_1\!\left(\frac12\right)$ is a separatrix; in particular, when written in the $(v,l_2)$-coordinates {with $l_2^2 =\frac{1}{r_1}$}, $\Gamma_1\!\left(\frac12\right)$ is a homoclinic orbit to the degenerate point $(v,l_2)=(0,0)$ and separates bounded orbits $\Gamma_1\!(h), h\in ]0,\frac12[$, from unbounded ones $\Gamma_1\!(h),h>\frac12$. See \figref{r1v}.

\begin{figure}[h!]
 	\begin{center}
 		{\includegraphics[width=.65\textwidth]{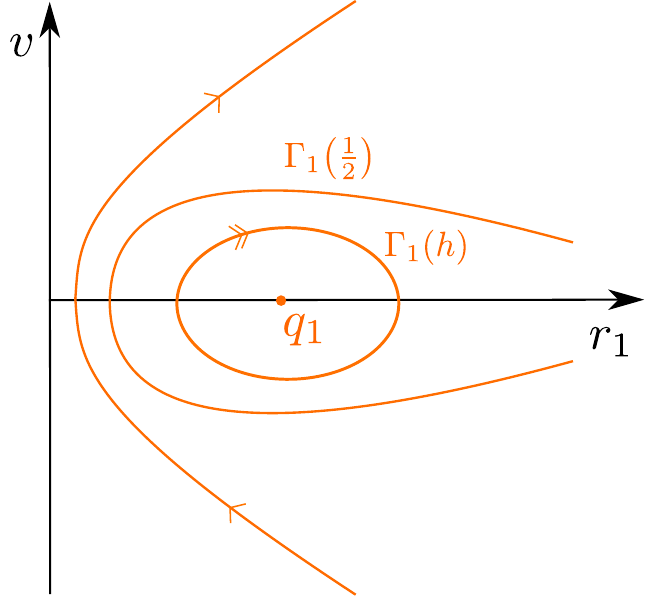}}
 		\caption{Phase portrait of \eqref{r1vHam}. The orbit $\Gamma_1\!\left(\frac12\right)$ defined by $H(r_1,v)=\frac12$ is a separatrix, separating bounded from unbounded orbits.}\figlab{r1v}
 	\end{center}
 \end{figure}


%
Notice that by \eqref{Enorm}, we have
\begin{align}
 H(r_1,v) = \frac12 e^2,\eqlab{HmathcalE}
\end{align}
and that $H(1,0)=0$ corresponds to zero eccentricity. 

{Finally, we emphasize that $(r_1,v,x)=(1,0,0)$ is a \textit{zero-Hopf point} of \eqref{V1eqnsgammaNegative1}, with the eigenvalues of the linearization being $\pm i$ and $0$. The center space of $(1,0,0)$ is three-dimensional.}

\begin{remark}
Notice that for the {undamped} Kepler problem, $\delta=0$, the Hamiltonian function appears naturally upon reduction to fixed angular momentum $l$.  
\end{remark}

Upon using \eqref{HmathcalE}, \propref{gammaNeg} follows from the following statement (see also \lemmaref{circ} {and the construction of the open set $\mathcal{U}$ in the previous section}).
\begin{proposition}\proplab{attractor}
 The point $(r_1,v,x)=(1,0,0)$ is an attractor for \eqref{V1eqnsgammaNegative1} on $x>0$.
\end{proposition}
\begin{proof}
 We will use the main result of \cite{bittmann2018a}, which we state in the following form:
 \begin{lemma}\lemmalab{bittmann}\cite[Theorem 1.10]{bittmann2018a} (Normal form of zero-Hopf bifurcation) Consider the analytic system 
 \begin{equation}\eqlab{xybittmann}
 \begin{aligned}
    z_1'&= -\lambda_0 (z_1 z_2) z_1 + x F_1(x,z_1,z_2),\\
  z_2'&= \lambda_0 (z_1 z_2) z_2 + x F_2(x,z_1,z_2),\\
   x' &=- x^2,
 \end{aligned}
 \end{equation}
 defined in a neighbourhood of the origin in $\mathbb{C}^3$,
 satisfying
 \begin{align}\eqlab{bittmancond}
\lambda_0(0)\in \mathbb C\backslash \{0\},\quad \operatorname{Re}(a_1+a_2)<0,
\end{align}
where
$$
a_1:=\frac{\partial F_1}{\partial z_1}(0,0),\quad a_2:=\frac{\partial F_2}{\partial z_2} (0,0), 
$$
%
Then there exist (a) $U\subset \mathbb C^2$ a neighborhood of $0$, (b) $\epsilon>0$, (c) a sector  
$$
 S = \left\{ x\in \mathbb C :\,0<\vert x\vert <\xi,\quad \lvert \operatorname{arg}\left(\frac{x}{i\lambda_0(0)}\right) \rvert<\frac{\pi}{2}+\epsilon\right\},
$$
and finally (d) an $x$-fibered analytic diffeomorphism $\Psi:U\times S\rightarrow \Psi(U\times S)\subset \mathbb C^3$, $\Psi (z,x)=(\psi (z,x),x)$ with $z=(z_1,z_2)$, satisfying: 
\begin{enumerate}
 \item \label{gevrey} $\Psi$ is (asymptotically) tangent to the identity at $(0,0)$, i.e. $$\Psi (z,x)=(z,x)+O(\vert (z,x)\vert^2),\; \; as\; (z,x)\to (0,0),\; (z,x)\in U\times S,$$
 and $\Psi$ and all of its partial derivatives extend continuously to $U\times \overline S$.
 \item $\Psi$ conjugates \eqref{xybittmann} with
 \begin{equation}\eqlab{bittmann2018a}
\begin{aligned}
     z_1' &= (-\lambda(z_1z_2)+x a_1)z_1,\\
   z_2' &=(\lambda(z_1z_2)+x a_2) z_2,\\
  x' &=- x^2,
 \end{aligned}
 \end{equation}
 with $\lambda$ analytic in a neighbourhood of $0$ and $\lambda(0)=\lambda_0(0)$.
\item \label{finall} Let $\overline z=\operatorname{Re}(z)-i\operatorname{Im}(z)$ denote $z\in \mathbb C$ conjugated. Suppose that $ z_2=\overline z_1,\, x\in S\cap \mathbb R_+$ defines an invariant set of \eqref{xybittmann}. Then $\lambda(z_1\overline z_1)\in i\mathbb R\backslash \{0\}$, $a_2=\overline a_1$ and $z_2=\overline z_1,\,x\in S\cap \mathbb R_+$ is also an invariant set of \eqref{bittmann2018a}. 
\end{enumerate}
%
   \end{lemma}
   The statement in (\ref{gevrey}), is weaker than in \cite{bittmann2018a}, but it is adequate for our purposes (in particular, the Gevrey-1 properties of $\Psi$ n \cite{bittmann2018a} is not needed here). 
   The statement in item (\ref{finall}) is not made explicit in \cite{bittmann2018a}, but it can be inspected directly from the proof of \cite[Theorem 1.10]{bittmann2018a}. In particular, the diffeomorphism $\Psi$ is constructed as a (weak) $1$-summable formal power series and it is easy to check that the formal map satisfies the statement. Moreover, we should emphasize that \cite[Theorem 1.10]{bittmann2018a} takes its point of departure from an analytic system defined on a neighborhood of $(x,z_1,z_2)=(0,0,0)$ with eigenvalues $\lambda_0,-\lambda_0,0$. In this more general version of the result, the quantities $a_1$ and $a_2$ are implicit. For this reason, we have therefore decided to state the result in the form \lemmaref{bittmann}, where $a_1$ and $a_2$ are explicit. (Notice, that $\Psi$ is not in general summable along the real axis). 

 For the convenience of the readers who are not familiar with \cite{bittmann2018a}, we give an outline on how to obtain \lemmaref{bittmann} from that paper. The vector field
$$Y=x^2 \frac{\partial}{\partial x}+(\lambda_0 (z_1 z_2 )z_1 -xF_1 (x,z))\frac{\partial}{\partial z_1}+(-\lambda_0 (z_1 z_2 )z_2 -xF_2 (x,z))\frac{\partial}{\partial z_2}$$
can be expressed as 
$$Y=x^2 \frac{\partial}{\partial x}+( \lambda_0 (0)z_1 -xF_1 (0,0)+\cdots )\frac{\partial}{\partial z_1}+(-\lambda_0 (0 )z_2 -xF_2 (0,0)+\cdots )\frac{\partial}{\partial z_2},$$
where the remainders are at least of degree two in $(x,z_1 ,z_2 )$. If $F_1 (0,0)=F_2 (0,0)=0$ this implies that $Y$ belongs to the class $\mathcal{SN}_{diag}$ with $\lambda =-\lambda_0 (0)$, \cite[Definition 1.1]{bittmann2018a}. In the general case, we introduce the change of variables
$$z_i =\tilde{z}_i +\alpha_i x,\; \; i=1,2,\; \; \alpha_1 =\frac{F_1 (0,0)}{\lambda_0 (0)},\; \alpha_2 =-\frac{F_2 (0,0)}{\lambda_0 (0)}.$$
From now on we assume that we are in the case $F_1 (0,0)=F_2 (0,0)=0$.
The matrix $A(x)$, obtained from the first order term of $F(x,\cdot )$, has the expansion
$$A(x)={\rm diag}(-\lambda_0 (0),\lambda_0 (0))-x\frac{\partial F(0,0)}{\partial z}+\cdots $$
Notice that this expansion is preserved by $(z_1 ,z_2 ,x)\mapsto
(\tilde{z}_1 ,\tilde{z}_2 ,x)$.
Then the residue of the vector field $Y$ is $-(a_1 +a_2 )$, see \cite[Definition 1.3]{bittmann2018a}, and the condition \eqref{bittmann2018a} implies that this residue has positive real part. This implies that $Y$ is \textit{strictly non-degenerate} and Theorem 1.5 in \cite{bittmann2018a} can be applied to find the change $\Psi$ as a formal series. A direct computation shows that the coefficients in the normal form can be expressed in terms of $a_1$ and $a_2$. To justify the convergence of the normal form and the change of variables we observe that the restriction of $Y$ to $x=0$ is orbitally linear. Indeed,
$$Y_{|x=0}=U(z)\left(-z_1 \frac{\partial}{\partial z_1}+z_2 \frac{\partial}{\partial z_2}\right),\; \, U(z)={\lambda_0 (z_1 z_2 )}\neq 0.$$
Then $Y$ is div-integrable and belong to the class $\mathcal{SN}_{diag,0}$, see \cite[Definition 1.9]{bittmann2018a}. Therefore \cite[Theorem 1.10]{bittmann2018a} is applicable.
%

   Consider \eqref{bittmann2018a} and suppose that $z_2=\overline z_1$ defines an invariant set. On this set, let $z_1=y_1+iy_2$, $z_2=y_1-iy_2$ with $y=(y_1,y_2)^T$ real, $\lambda(z_1\overline z_1)=:i\Omega(z_1\overline z_1)$ with $\Omega$ real valued and $a_1=a-ib$, $a_2 =\overline{a}_1$. Then 
   \begin{equation}\eqlab{bittmanfinal}
   \begin{aligned}
      \dot y &= 
\begin{pmatrix}
a x & \Omega(y_1^2+y_2^2)+b x \\
 -\Omega(y_1^2+y_2^2)-b x & ax
\end{pmatrix}y,\\
   \dot x &=-x^2.
   \end{aligned}
   \end{equation}
    \begin{cor}\corlab{bittmann}
    Consider \eqref{bittmanfinal} and assume $a<0$. Let $(y_1(\tau),y_2(\tau),x(\tau))$ denote a solution with initial condition $x(0)\in S\cap \mathbb R_+$, $(y_1(0),y_2(0))\in B$ with $B$ being a small ball centered at $(0,0)$.  Then this solution is defined in $[0,\infty )$ and
    $$
     y_1(\tau)^2+y_2(\tau)^2 = \frac{A x(\tau)^{-2a} }{x(0)^{-2a} } =\frac{A}{(1+x(0)\tau)^{-2a}},
   $$
   where $A=y_1(0)^2+y_2(0)^2$.
   \end{cor}
 \begin{proof}
    This follows from the equality
   $$
     \frac{d}{d\tau}(y_1^2+y_2^2) = 2ax (y_1^2+y_2^2)<0.
   $$
for $x\in S\cap \mathbb R_+$. In fact $B\times (S\cap \mathbb R_+)$ is positively invariant and $x(\tau)=\frac{x(0)}{1+x(0)\tau}$, $\tau\ge 0$.
\end{proof}

   To complete the proof of \propref{attractor}, we therefore bring \eqref{V1eqnsgammaNegative1} into the form \eqref{xybittmann}. \corref{bittmann}, and the continuity of {$\Psi^{-1}$} at the origin, see item (\ref{gevrey}) in  \lemmaref{bittmann}, then gives the desired result. 
   
    To find the corresponding changes of variables we proceed in two steps. First, we divide the right hand side of \eqref{V1eqnsgammaNegative1} by the quantity 
  $$
   U(r_1,v):= \tilde \gamma \delta  r_1^{-\alpha-\beta} \left(1+r_1^2v^2\right)^{\frac{\alpha}{2}}>0,
  $$
satisfying $U(r_1,-v)=U(r_1,v)$. This therefore corresponds to a regular transformation of time. Subsequently, we bring the resulting $x=0$ subsystem:
\begin{equation}\label{rev}
\begin{aligned}
 r_1' &=\frac{1}{U(r_1,v)}\frac{\partial H}{\partial v}(r_1,v),\\
 v' &=-\frac{1}{U(r_1,v)}\frac{\partial H}{\partial r_1}(r_1,v),
\end{aligned}
\end{equation}
in a neighborhood of $(r_1,v)=(1,0)$, into the Birkhoff normal form:
\begin{equation}\nonumber
 \begin{aligned}
  \dot y_1 &= \Omega_0(y_1^2+ y_2^2) y_2,\\
  \dot y_2 &=-\Omega_0(y_1^2+ y_2^2) y_1.
 \end{aligned}
 \end{equation}
 $\Omega_0$ is analytic, 
 see e.g. \cite[Corollary 4]{delshams2002a} and notice that (\ref{rev})
is reversible: $(r_1,v)(t)$ is a solution $\Rightarrow$ $(r_1,-v)(-t)$ is a solution.

For future computations, we need the quadratic terms of the transformation $(y_1 ,y_2 )\mapsto (r_1 ,v)$. First we expand the system (\ref{rev}) at $(1,0)$,
\begin{align*}
r_1' &=\frac{v}{\tilde \gamma\delta} +\frac{\alpha +\beta}{\tilde \gamma\delta}v(r_1 -1)+\cdots,\\
v'&=-\frac{r_1 -1}{\tilde \gamma\delta}+\frac{3-\alpha -\beta}{\tilde \gamma\delta}(r_1 -1)^2 +\cdots,
\end{align*}
 and then upon writing $z_1=y_1+iy_2$, $z_2=y_1-iy_2$ and $w_1=r_1-1+v$, $w_2=r_1 -1-iv$, we obtain the system
 \begin{align*}
 z_1'&= -\lambda_0 (z_1 z_2)z_1,\\
 z_2'&= \lambda_0 (z_1 z_2)z_2.
 \end{align*}
 Here
 $$ \lambda_0(z_1z_2) = i\Omega_0(z_1z_2),\quad \lambda_0(0)= \frac{i}{\delta \tilde \gamma}$$
 and
 \begin{align*}
 w_1'&=-\lambda_0 (0)w_1 +P_1 (w_1 ,w_2 ),\\ 
 w_2'&= \lambda_0 (0)w_2 +P_2 (w_1 ,w_2 ),
 \end{align*}
 with $$P_1 (w_1 ,w_2 )=\lambda_0 (0)\left[\frac{1}{4}(3-2\alpha -2\beta )w_1^2 +\frac{1}{2}(3- \alpha - \beta )w_1 w_2 +\frac{3}{4}w_2^2 +\cdots \right]$$
 and $P_2  (w_1 ,w_2 )=\overline P_1 (\overline{w}_2 , \overline{w}_1)$.
 
 Now, let ${Q}(z)=\mathcal O(\vert (z_1,z_2)\vert^2)$ be so that
 \begin{align*}
  w = z+{Q}(z).
 \end{align*}
Then we find that ${Q}=({Q}_1,{Q}_2)^T$  has the following expansion
 $${Q}_1 (z_1 ,z_2 )=-\frac{1}{4}(3-2\alpha -2\beta )z_1^2+\frac{1}{2} (3-\alpha -\beta )z_1z_2 +\frac{1}{4}z_2^2 +\mathcal O(\vert (z_1,z_2)\vert^3),$$
with ${Q}_2$ linked to ${Q}_1$ through conjugacy. 
 
   We now return to the three-dimensional system. In the new variables, $w_1$, $w_2$ and $x$, we obtain
 \begin{align*}w_1'&=\frac{1}{U} \left(\frac{\partial H}{\partial v}-i\frac{\partial H}{\partial r_1}\right)+x\frac{2}{\tilde \gamma}(w_2 +1),\\
 w_2'&=\frac{1}{U} \left(\frac{\partial H}{\partial v}+i\frac{\partial H}{\partial r_1}\right)+x\frac{2}{\tilde \gamma}(w_1 +1).
 \end{align*}
 The change of coordinates defined by $w=z+Q(z)$, $w=(w_1 ,w_2 )^T$, $z=(z_1 ,z_2 )^T$ leads to
 $$z'=\Lambda (z_1 z_2 )z+x\mathcal{F}(z)$$
 where $$\Lambda =\operatorname{diag}(-\lambda_0 ,\lambda_0 ),\quad S=\begin{pmatrix}
0 & 1 \\
1 & 0
\end{pmatrix},$$ and
 $$\mathcal{F}(z)=\frac{2}{{\tilde \gamma}}(I+DQ(z))^{-1}[S(z+Q(z))+(1,1)^T],$$
 {with $DQ$ denoting the Jacobian of $Q=(Q_1,Q_2)$}.
 From the above computations, it follows that
 $$D\mathcal{F}(0)=\left(
\begin{array}{cc}
-\frac{\alpha +\beta}{{\tilde \gamma}} & \star \\
\star & -\frac{\alpha +\beta}{{\tilde \gamma}}
\end{array}
\right).$$
The system satisfies all the conditions of \lemmaref{bittmann}, in particular
\begin{align}
a=a_1=a_2=-\frac{\alpha +\beta}{\tilde \gamma}<0,\eqlab{a1a2}
\end{align}
for all $0<\tilde \gamma<3$. This completes the proof of \propref{attractor}.
 \end{proof}
 \begin{lemma}\lemmalab{timeblowup2}
  Suppose that $-3<\gamma<0$ and that $u(t)$ is a solution of \eqref{keplerd} where $e(t)\rightarrow 0$ as $t\rightarrow \omega$. Then $\omega<\infty$.
 \end{lemma}
\begin{proof}
We know from the circularization assumption, \lemmaref{circ} and \corref{circ1} that the corresponding solution of \eqref{V1eqnsgammaNegative1} is defined on $[0,\infty )$ and converges to the equilibrium.
 For this we use \eqref{trans2}. 
 The $x$-equation of 
 \eqref{V1eqnsgammaNegative1}, is repeated here for convenience:
 \begin{align}\eqlab{xeqn}
   x'&=-\tilde \gamma \delta  r_1^{-\alpha-\beta} \left(1+r_1^2v^2\right)^{\frac{\alpha}{2}}x^2.
 \end{align}
 Recall that 
 \begin{align}
 \tilde \gamma = -\gamma\in (0,3).\eqlab{tildegamma}
 \end{align}
 Now, $e(t)\rightarrow 0$ implies that $r_1(\tau)\rightarrow 1$ and $v(\tau)\rightarrow 0$. 
 Therefore we can estimate $x$ as follows
 \begin{align*}
  x \sim \frac{1}{\tau},
 \end{align*}
for all $\tau\ge  1$, say. 
%
 By this estimate, together with \eqref{trans2} and \eqref{tildegamma}, we conclude that
\begin{equation}\eqlab{Cauchy} {\omega =\int_0^{\infty} \frac{dt}{d\tau} (\tau )d\tau =\int_0^\infty x(\tau )^{\frac{3}{\tilde \gamma}} d\tau <\infty.}\end{equation}
\end{proof}

\begin{lemma}\lemmalab{p2}
Consider the same assumptions as in \lemmaref{timeblowup2}. Then 
\begin{align*}
 \limsup_{\tau \rightarrow \infty} \vert p(\tau)\vert  \begin{cases}
                     =\infty & \mbox{for $\alpha+\beta-1>0$},\\
                     \in (0,\infty) & \mbox{for $\alpha+\beta-1=0$},\\
                     =0 & \mbox{for $\alpha+\beta-1<0$}.
                    \end{cases}
\end{align*}

\end{lemma}
\begin{proof}
 We use \corref{bittmann} with $a=-\frac{\alpha+\beta}{\tilde \gamma}$, see \eqref{a1a2}, and the properties of the mapping $\Phi$ in \lemmaref{bittmann} to conclude that 
 \begin{align*}
  (r_1(\tau)-1)^2 + v(\tau)^2 \sim \tau^{-\frac{2(\alpha+\beta)}{\tilde \gamma}}.
 \end{align*}
 In particular,
 \begin{align*}
  \limsup_{\tau\rightarrow \infty} \vert v(\tau) \tau^{\frac{\alpha+\beta}{\tilde \gamma}}\vert <\infty.
 \end{align*}
Therefore upon using $x(\tau)\sim \tau^{-1}$ and \eqref{trans2}, we conclude that 
\begin{align*}
 \limsup_{\tau\rightarrow \infty} \vert p(\tau) \tau^{\frac{\alpha+\beta-1}{\tilde \gamma}}\vert <\infty.
\end{align*}
The result therefore follows. 
\end{proof}
\begin{lemma}\lemmalab{theta2}
Consider the same assumptions as in \lemmaref{timeblowup2}. Then 
\begin{align*}
 \theta(t) \rightarrow \infty,
\end{align*}
for $t\rightarrow \omega$. 
\end{lemma}
\begin{proof}
 As in the proof of \lemmaref{theta1}, we use $l=r^2 \dot{\theta}$, $\dot\theta=\frac{d}{dt}\theta$. In particular, in combination with \eqref{trans2} and \eqref{Cauchy}, we find that 
$$\frac{d\theta}{d\tau} =\frac{d\theta}{dt} \frac{dt}{d\tau}=\frac{1}{r_1^2}\to 1,$$
for $\tau\rightarrow \infty$. 
Consequently, the limit
$$\lim_{t\to  \omega^-} \theta (t)=\theta (0)+\int_0^\infty \frac{d\theta}{d\tau}(\tau )d\tau,$$ is infinite.
\end{proof}

 \section{The critical case $\gamma=0$}\seclab{gammaEq0}
  Finally, in the case $\gamma=0$, we prove the following
  \begin{proposition}\proplab{gamma0}
Consider \eqref{rvl} for $\gamma=0$, $\delta>0$. {Then there is no circularization for any initial condition. Also, for each $\delta >0$ there exists a non-empty open set of initial conditions such that ($\lim_{t\rightarrow \omega} e(t) = 4\delta^2$ if $\delta\in \left(0,\frac12\right)$) and ($\lim_{t\rightarrow \omega} e(t) =1$ if $\delta\ge \frac12$). }
\end{proposition}
 To prove \propref{gamma0}, we follow \secref{blowup}, specifically \secref{blowupNeg}, and use a change of variables that can be summed up by the formulas:
\begin{align}
 r= \rho_1^2 r_1,\quad p=\frac{v}{\rho_1},\quad l=\rho_1,\quad \frac{d\tau}{dt}=\frac{1}{\rho_1^{3}}.\eqlab{transcrit}
\end{align}
 This gives 
\begin{equation}\eqlab{X1}
\begin{aligned}
 r_1 '& = v + 2\delta  r_1^{-\alpha-\beta} \left(1+r_1^2v^2\right)^{\frac{\alpha}{2}}r_1,\\
 v'&=-\frac{r_1-1}{r_1^3}-2\delta r_1^{-\alpha-\beta} \left(1+r_1^2v^2\right)^{\frac{\alpha}{2}}v,
\end{aligned}
\end{equation}
and $\rho_1'=- \delta r_1^{-\alpha-\beta} \left(1+r_1^2v^2\right)^{\frac{\alpha}{2}}\rho_1$, where $()'=\frac{d}{d\tau}$. Notice that in this critical case, $\rho_1$ decouples; this relates directly to \lemmaref{critical}. 

The absence of circularization is now easily proved. In fact, the existence of some initial condition with the circularization property would imply that $(r_1 (t),v(t))\to (1,0)$ as $t\to \omega$ for some solution of \eqref{X1}. This is impossible since the point $(1,0)$ is not an equilibrium. 

Let $W_1$ denote the planar vectorfield defined by \eqref{X1}. Then
an easy computation shows that the divergence of $W_1$ is given by 
\begin{align}\eqlab{div}
 \operatorname{div} W_1(r_1,v) = -2(\alpha+\beta)r_1^{-\alpha-\beta} \left(1+r_1^2v^2\right)^{\frac{\alpha}{2}}.
\end{align}
Since this quantity is strictly negative for $r_1>0$, it follows from Dulac's theorem that $W_1$ has no limit cycles within $r_1>0$. 

\begin{lemma}
Suppose that $\gamma=0$. Then there exists an equilibrium of \eqref{X1} with $r_1>0$ if and only if $\delta\in\left(0,\frac12\right)$. In the affirmative case, this equilibrium is hyperbolically attracting and given by 
 \begin{align}\eqlab{eq}
  r_1 &=\frac{1}{1-4\delta^2},\quad v =-2\delta \sqrt{1-4\delta^2}.
 \end{align}
 Moreover, if $\delta\in \left(0,\frac12\right)$ then $e$, given by \eqref{HmathcalE}, equals $$e =4  \delta^2,$$
 along
  \eqref{eq}.
\end{lemma}
\begin{proof}
From the $r_1'=0$-equation with $r_1>0$, we first realize that $v<0$.
 Then we consider $\frac{r_1' v}{r_1}+v'=0$ which leads to
 \begin{align*}
  v =- r_1^{-\frac{1}{2}} \left(1-r_1^{-\frac{3}{2}}\right)^{\frac12}.
 \end{align*}
Inserting this into $r'=0$ gives
\begin{align*}
 2\delta r_1-\sqrt{r_1(r_1-1)} =0.
\end{align*}
This equation has a solution 
\begin{align*}
 r_1 &= \frac{1}{1-4\delta^2}>1,
\end{align*}
if and only if $\delta\in \left(0,\frac12\right)$.
The statement regarding the existence of the equilibrium follows. Now regarding the stability, we compute the determinant of the Jacobian $J$ of $W_1$ at the equilibrium. After some algebraic manipulations, we find that 
\begin{align*}
 \operatorname{det} J=(2\delta +1)^4(1-2\delta)^4>0.
\end{align*}
Given that the trace of $J$ is negative, see \eqref{div}, we conclude that the equilibrium is hyperbolic and either a stable node or a stable focus.
\end{proof}
\begin{remark}
    Upon using \eqref{transcrit}, we realize that the equilibrium \eqref{eq} of \eqref{X1} corresponds to a solution $(r(t),p(t),l(t))$ of \eqref{rvl0}
    \begin{align*}
        r(t) = l(t)^2 \frac{1}{1-4\delta^2},\quad p(t) = r'(t)=-2l(t)^{-1}\delta \sqrt{1-4\delta^2},
    \end{align*}
    where 
    \begin{align*}
       l(t) = \left(l(0)^3-3\delta (1-4\delta^2)^{\frac32}t\right)^{\frac13}.
    \end{align*}
    with $t\in [0,\omega)$, $\omega = \frac{l(0)^3}{3\delta (1-4\delta^2)^{\frac32}}$.  

We emphasize that on the basis of \eqref{eq}, $\mathcal E$ does not have a limit. Indeed, we have $\theta(t)\rightarrow \infty$ as $t\rightarrow \omega$, by proceeding as in \lemmaref{theta2}, and $e(t)\rightarrow 4\delta^2$. Note that $\mathcal{E}=-vie^{i\theta} +(\frac{l^2}{r}-1)e^{i\theta}$.
\end{remark}
The claim for the case $\delta\in \left(0,\frac12\right)$ follows easily. To analyze the remaining case we introduce a new change of variables, defined by
$$r_1 =\frac{1}{\mu_1^2},\; \; v=\mu_1 v_1 ,\; \; \rho_2 =\rho_1\sqrt{ r_1},\; \; \frac{d\tau}{ds} =\mu_1^{-3}.$$
In particular, 
\begin{align}\eqlab{rrho2}
r=\rho_2^2.
\end{align}
This gives the following equations
\begin{equation}\eqlab{v1mu1}
\begin{aligned}
    v_1' &=\frac12 v_1^2 - 1 + \mu_1^2 - \delta (\mu_1^2 + v_1^2)^{\frac{\alpha}{2}}v_1,\\
    \mu_1 '&=- \left(\frac12 v_1+\delta(\mu_1^2 + v_1^2)^{\frac{\alpha}{2}} \right)\mu_1,
\end{aligned}
\end{equation}
and $\rho_2' =\frac12 \rho_2 v_1$. Now, $\rho_2$ decouples. We show the phase portrait of \eqref{v1mu1} (which is equivalent to the phase portrait of \eqref{X1} for $\mu_1>0$) for $\delta=0.2$ and $\alpha=1$ in \figref{mu1v1critical}. \begin{figure}[h!]
 	\begin{center}
 		{\includegraphics[width=.6\textwidth]{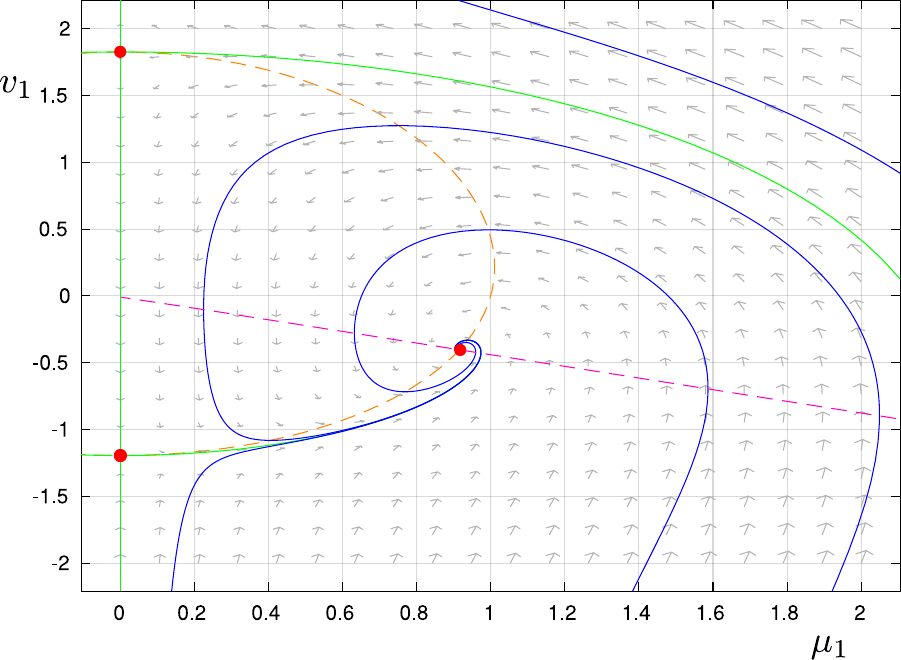}}
 		\caption{Phase portrait of \eqref{v1mu1}  for $\delta=0.2$ and $\alpha=1$. There are three equilibria (red dots) in this case, two of saddle-type along $\mu_1=0$ and one stable (focus type) within $\mu_1>0$. The equilibrium within $\mu_1>0$ corresponds to \eqref{eq}. The green orbits are invariant manifolds, while the dotted lines (orange and magenta) are the nullclines. }\figlab{mu1v1critical}
 	\end{center}
 \end{figure}

{It is easy to see that there is unique negative solution $v_{1,*}<0$ of the following equation for $v_1$
$$
    \frac12 v_1^2 = 1+\delta |v_1|^{\alpha}v_1.
$$
 In turn, we obtain a unique equilibrium $(v_{1,*},0)$ of \eqref{v1mu1} within $v_1\le 0, \mu_1=0$. }
This equilibrium is asymptotically stable whenever $\delta \geq \frac{1}{2}$ (it is a saddle for $0<\delta<\frac12$, see \figref{mu1v1critical}). This is proved by linearization if $\delta >\frac{1}{2}$ and using a center manifold computation if $\delta =\frac{1}{2}$. The conclusion follows from the expression of the eccentricity in terms of the variables $r_1$ and $v$.

\begin{remark}\remlab{r1zerocrit}
In the critical case, we can relatively easy show that the two equilibria: \eqref{eq} of \eqref{X1} for $\delta\in (0,\frac12)$ and $(v_{1,*},0)$ of \eqref{v1mu1} for $\delta\ge \frac12$, define the omega limit sets of \eqref{keplerd} on the set of collisions, i.e. $u(t)\rightarrow 0$. For simplicity, we restrict attention to $\delta\ge \frac12$ case and consider the equations \eqref{v1mu1} together with 
\begin{align}\eqlab{rho1eqn}
\rho_2'=\frac12 \rho_2v_1.
\end{align}
\begin{lemma}
Consider $\gamma=0$ and $\delta\ge \frac12$ and suppose that $u(t)\rightarrow 0$. Then $(v_1(\tau),\mu_1(\tau))\rightarrow (v_{1,*},0)$.
\end{lemma}
\begin{proof}
Since $u(t)\rightarrow 0$, we have that $\rho_2(\tau)\rightarrow 0$, see \eqref{rrho2}.
By \eqref{div} there can be no limit cycles (\eqref{v1mu1} is equivalent with \eqref{X1}). Consequently, it follows from Poincar\'e-Bendixson that either $(v_1(\tau),\mu_1(\tau))$ limits to an equilibrium as $\tau\rightarrow \infty$ or $(v_1(\tau),\mu_1(\tau))$ goes unbounded. $(v_{1,*},0)$ is the unique equilibrium within $v_1\le 0$ and we have already established that it is asymptotically stable for $\delta\geq \frac12$. All other equilibria are contained within $v_1> 0$ and consequently we would have $\rho_2(\tau)\rightarrow \infty$ cf. \eqref{rho1eqn}. Therefore, in order to complete the proof, it remains to exclude the possibility: 
\begin{align}\eqlab{final2}
\mbox{$(v_1(\tau),\mu_1(\tau))$ goes unbounded and $\rho_2(\tau)\rightarrow 0$.}
\end{align}

Let $G$ denote the arc of the ellipse
\begin{align*}
\frac12 v_1^2-1+\mu_1^2 = 0
\end{align*}
within $v_1\le 0$. It intersects $\mu_1=0$ in the point $(v_1,\mu_1)=(-\sqrt{2},0)$ and $v_1=0$ in a point $(v_1,\mu_1)=(0,1)$. Next, let $H$ denote the interval $(1,\infty)$ on the $\mu_1$-axis. Finally, we  put $F:=G\cup H$. 

It should be clear from \eqref{v1mu1} that $F$  bounds a region below where $v_1'>0$ and $\mu_1\ge 0$. If \eqref{final2} holds then $(v_1(\tau),\mu_1(\tau))$ is eventually, i.e for all $\tau>0$ large enough, contained in the region below $F$. Indeed, if this were not the case, then we can again arrive at the contradiction $\rho_2(\tau)\rightarrow \infty$ cf. \eqref{rho1eqn}. Here we have also used that the $v_1$-axis is invariant. {Note also the following property of the system: orbits cannot cross $H$ at two points. Otherwise the arc of the orbit connecting these two points together with the segment joining them in $H$ would define a negatively {or positively} invariant region without any equilibrium.} Since $v_1'(\tau)>0$ in the region below $F$, $v_1(\tau)$ is contained in an interval $[-c,0)$, $c>0$, and \eqref{final2} therefore implies that $v_1(\tau)\rightarrow 0^-,\mu_1(\tau)\rightarrow \infty$. But this is impossible: We have $\mu_1'<0$ for all 
\begin{align*}
   \mu_1\ge \left(\frac{c}{2\delta} \right)^{\frac{1}{\alpha}},
\end{align*}
using $v(\tau)\in [-c,0)$ and \eqref{v1mu1}.
\end{proof}

\end{remark} 
\section{Completing the proof of \thmref{main}}\seclab{completing}
The statement of \thmref{main} item (1) regarding circularization follows from \propref{gammaPos}, \propref{gammaNeg} and \propref{gamma0}.  
In particular, (1a) follows from \propref{gammaPos}, \lemmaref{p1} and \lemmaref{theta1}.
Moreover, (1b) follows from \propref{gammaNeg}, \lemmaref{p2} and \lemmaref{theta2}.
Finally, (2) follows from 
\lemmaref{timeblowup1} and \lemmaref{timeblowup2}. 

\section{Discussion}\seclab{discussion}
{In the present paper, we have provided a complete description of circularization within a general class of dissipation forces, see \eqref{Dfunc}; this class goes all the way back to See and Poincar\'e, \cite{lecons,See}. In contrast to Poincar\'e in \cite{lecons}, who claimed that circularization occurs for all $\alpha$ and $\beta$ sufficiently large, we have shown that circularization (on an open set of initial conditions) only occurs for $-3<\gamma=\alpha+2\beta-3<0$. Interestingly, circularization for  $-3<\gamma<0$ is accompanied by finite time blowup $\omega<\infty$, see \figref{diagram} where the main results are summarized. In other words, \textit{circularization and global time existence is impossible for \eqref{keplerd} within the class \eqref{Dfunc}, $\alpha\ge 0$, $\beta\ge 0$, $(\alpha,\beta)\ne (0,0)$}, see also \lemmaref{timeblowup2}.

Our approach was based on desingularization and blowup. It is our expectation that this approach, which was also successful in \cite{kristiansen2023revisiting} for the description of the case of linear damping $\alpha=\beta=0$, can be used to address some of the remaining open questions. In particular, in future work we aim to provide a complete characterization of unbounded solutions. We already have some partial results in this direction. At the same time, we also plan to present a full proof of the conjecture that the open set in \thmref{main} coincides with the set for which $u(t)\rightarrow 0$. We only addressed this in the critical case, see \remref{r1zerocrit}.}

\bibliography{refs}
\bibliographystyle{plain}

\end{document}